\documentclass[12pt]{amsart}
\usepackage{amsmath,amssymb,latexsym}
\usepackage{fullpage}

\theoremstyle{plain}
\newtheorem{theorem}{Theorem}
\newtheorem{proposition}{Proposition}
\newtheorem{lemma}{Lemma}
\newtheorem{definition}{Definition}
\newtheorem{conjecture}{Conjecture}
\newtheorem{problem}{Problem}

\begin{document}
\date{}

\title[Classification of some Global Integrals]
{\bf Classification of some Global Integrals related to groups of
type $A_n$ }
\author{ David Ginzburg}

\begin{abstract}
In this paper we start a classification of certain global integrals. First, we use the language of unipotent orbits to write down a family of global integrals. We then classify all those integrals which satisfy the dimension equation we set. After doing so, we check which of these integrals are global unipotent integrals. We do all this for groups of type $A_n$, and using all this we derive a certain interesting conjecture about the length of these integrals.

\end{abstract}

\thanks{ The author is partly supported by the Israel Science
Foundation grant number  259/14}

\address{ School of Mathematical Sciences\\
Sackler Faculty of Exact Sciences\\ Tel-Aviv University, Israel
69978 } \maketitle \baselineskip=18pt

\section{\bf Introduction}

In this paper we begin a classification of what we refer to as
global unipotent integrals. Constructing global integrals is one of
the ways in which one can study the Langlands conjectures related to
$L$ functions. In this method, one constructs a global integral
which depends on a complex variable $s$, and the goal is to determine
if this integral is Eulerian. To do so one carries out the process
of unfolding, which consists mainly on a series of Fourier
expansions. At the end we obtain an integral which involves an
integration over a quotient of the type $N({\bf A})\backslash G({\bf
A})$ where $N$ and $G$ are two groups. The integral involves a set
of functionals or bi-linear forms etc. Then the task is to determine if
this integral is factorizable. This is the case if, for example, the
functional or the bi-linear forms etc.  which appear in the
integral, satisfies some uniqueness properties. Maybe the most
difficult part in this program  is to determine a way of how to
actually construct the initial global integral. In \cite{G1} we
describe a way using the language of unipotent orbits to construct
such global integrals. One of the most general construction is given
by integral \eqref{global1} which appears in the next section.
However, as mentioned at the end of that section, this is not the
most general construction. Nevertheless, the vast majority of global
Eulerian integrals which appear in the literature are of the type of
integral \eqref{global1}.

Thus, one of the main problems  is to determine all global integrals
\eqref{global1} which, for $\text{Re}(s)$ large, after the unfolding
process, unfold to the global integral \eqref{global2}. We refer to
such integrals as global unipotent integrals. As mentioned above,
the process of unfolding involves mainly certain Fourier expansions.
Therefore, a good knowledge of the Fourier coefficients of the
representations in question, is crucial. In our context, for a
representation $\pi$, this is captured by the set of unipotent
orbits ${\mathcal O}(\pi)$. This notion  is defined in \cite{G2}
definition 2.1. This, in turn, leads to the definition of the
dimension of the representation $\pi$, and to the definition of what
we refer to as the dimension equation of a global integral. See
equation \eqref{dim1}. Roughly speaking, this equation states that
the sum of the dimensions of the representations involved in the
integral, is equal to the sum of the dimensions of the  groups which
are involved in the integration. There are several reasons which
motivate  to set up this dimension equation. Maybe the main reason
is simply because all known integrals which unfold to integrals involving the Whittaker
coefficient of at least some of the representations, do
satisfy this equation. See \cite{G1} and \cite{G3}.

To summarize, a main aspect of this theory is to classify all
nonzero global unipotent integrals which are given by integral
\eqref{global1}, and which satisfies the dimension equation
\eqref{dim1}.

In this paper we consider this classification for the group
$G=GL_m$. We do not do it for the most general case, but rather
restrict things to the case where the embedding of the group $G$
inside the groups $G_j$, see section 2, is such that the center of
$G$ coincides with the center of $G_j$. In section 3 we list all
possible Fourier coefficients, defined on an arbitrary reductive
group $H$, such that the stabilizer of this coefficient inside some
Levi part of a certain parabolic subgroup, is the group $G$ embedded in $H$
as mentioned above. Once we classify all such Fourier coefficients,
we obtain a list of global integrals defined by integral
\eqref{global1}. The first step is to write down the dimension
equation for these integrals, and to study it. This is done in
section 4. There are two main consequences which arise from the
study of this equation. The first one is a list of global integrals
which, in the cases of $m=2,3$,  is a complete list of all possible
integrals. For $m=2$ this list is given in Tables 1 and 2, and for
$m=3$ it is given in Tables 3-7. When $m\ge 4$ we only get a partial
list, given in Table 8. We emphasize that this Tables list  global
integrals of the type of integral \eqref{global1} which satisfies
the dimension equation \eqref{dim1}. It does not guarantee that these
integrals are nonzero, and if nonzero that they are a global {\sl
unipotent} integral.

An interesting result regarding what we define as the length of the integral, follows
from these Tables. Given a global integral of the type of integral
\eqref{global1}, we define its length to be the number of
representations involved in the integral. Conjecture \ref{conj1},
states that for all $m\ge 2$, if a global integral is a nonzero
global unipotent integral with $G=GL_m$, then its length is at most
three. In section 4 we prove this conjecture for $m=2,3$ and
preliminary computations indicate that this conjecture is true for
all $m$. We hope to study this problem in the near future. We should
mention that while this conjecture is obvious for $m=2$, for $m=3$
it requires a nontrivial result about cuspidal representations of
the exceptional group $E_6$. This result, is stated in lemma
\ref{lem1} and proved in the last section of the paper.

Sections 5,6 and 7 consists of unfolding the global integrals for
the case when $m=2$. As mentioned above, the lists we obtain in
section 4 do not guarantee that the global integral in question is nonzero, or
if it is a global unipotent integral. For that we need the process of
the unfolding. After some preparation which are done in sections 5
and 6, in section 7 we prove two Theorems which states in which
cases the global integral is actually a nonzero global unipotent
integral. For that we define the notion of an odd Eisenstein
series, and prove in Theorems \ref{th1} and \ref{th2} that a global
integral which appears in Tables 1 or 2 is a nonzero global unipotent
integral if and only if at least one of the representations involved in the
integral, is an odd Eisenstein series.

\section{\bf The Basic Setup}

Let ${\bf A}$ be the ring of adeles of a global field $F$. Let
$\psi$ denote a nontrivial additive character of $F\backslash {\bf
A}$. For basic facts and notation about unipotent orbits we refer to
\cite{C} and \cite{C-M}.

We first recall some basic facts about unipotent orbits. As
explained in \cite{G2} section 2, given a reductive group $H$, and a
unipotent orbit ${\mathcal O}$ of $H$, one can associate with this
orbit a set of Fourier coefficients. Thus, to ${\mathcal O}$, we can
associate a certain unipotent subgroup $U({\mathcal O})$ of $H$, and
a character $\psi_{U({\mathcal O})}$ of $U({\mathcal
O})(F)\backslash U({\mathcal O})({\bf A})$. It is possible that to a
given unipotent orbit there will correspond infinite number of
Fourier coefficients. Hence the choice of the character
$\psi_{U({\mathcal O})}$ is not always unique.

Given an irreducible automorphic representation $\pi$ of $H({\bf
A})$, one can associate with it a set of unipotent orbits of $H$,
which we denote by ${\mathcal O}_{H}(\pi)$. This set is defined in
\cite{G2} definition 2.1. As mentioned in that reference, it is
conjectured that this set consists of a unique unipotent orbit.
Henceforth, we shall assume that this is the case. Another notion we
need is the notion of the dimension of $\pi$. We define $\text{dim}\
\pi=\frac{1}{2}\text{dim}\ {\mathcal O}_{H}(\pi)$. Notice that this
notion is well defined only if we assume that ${\mathcal
O}_{H}(\pi)$ consists of one unipotent orbit.

For $1\le i\le l$, let $G_i$ denote $l$ reductive groups. For $1\le
j\le l$, let $\pi_j$ denote an automorphic representation defined on
$G_j({\bf A})$. Let ${\mathcal O}_{G_i}$ denote a unipotent orbit of
the group $G_i$. As explained above, we let $U({\mathcal O}_{G_i})$
denote the corresponding unipotent group, and we choose a corresponding character
$\psi_{U({\mathcal O}_{G_i})}$ of $U({\mathcal
O}_{G_i})(F)\backslash U({\mathcal O}_{G_i})({\bf A})$. Assume, that
the stabilizer of $\psi_{U({\mathcal O}_{G_i})}$ inside a suitable
Levi subgroup of $G_i$ contains the {\sl same} reductive group $G$. Then
we can form the global integral
\begin{equation}\label{global0}
\int\limits_{Z({\bf A})G(F)\backslash G({\bf
A})}\varphi_{\pi_1}^{U_1,\psi_{U_1}}(g)
\varphi_{\pi_2}^{U_2,\psi_{U_2}}(g)\ldots
\varphi_{\pi_{l-1}}^{U_{l-1},\psi_{U_{l-1}}}(g)
\varphi_{\pi_{l}}^{U_l,\psi_{U_l}}(g)dg
\end{equation}
Here $Z$ is the center of $G$, and we assume that the embedding of $G$ is such that we can divide by $Z$. For each $i$ we have
$$\varphi_{\pi_i}^{U_i,\psi_{U_i}}(g)=
\int\limits_{U({\mathcal O}_{G_i})(F)\backslash U({\mathcal
O}_{G_i})({\bf A})}\varphi_{\pi_i}(ug)\psi_{U({\mathcal
O}_{G_i})}(u)du$$
Allowing the vectors $\varphi_{\pi_i}$ to vary in the space of the representation $\pi_i$, the above set of Fourier coefficients defines an automorphic representation $\sigma_i$ of $G({\bf A})$. Henceforth, we will assume that for all $i$, the representation $\sigma_i$ is not a one dimensional representation.

We assume that the representations $\pi_j$ are
such  that the integral \eqref{global0} converges. This will be the case if we
assume that one of the representations $\pi_i$ is an irreducible
cuspidal representation. After reordering, we may assume that
$\pi_1$ is cuspidal. We also want the above integral to depend on a
complex variable $s$. To do so, we assume that $\pi_l=E_\tau$ is an
Eisenstein series defined on the group $G_l({\bf A})$. Thus, the
global integral we study is given by
\begin{equation}\label{global1}
\int\limits_{Z({\bf A})G(F)\backslash G({\bf
A})}\varphi_{\pi_1}^{U_1,\psi_{U_1}}(g)\varphi_{\pi_2}^{U_2,\psi_{U_2}}(g)\ldots
\varphi_{\pi_{l-1}}^{U_{l-1},\psi_{U_{l-1}}}(g)E_\tau^{U_l,\psi_{U_l}}(g,s)dg
\end{equation}
As mentioned above, and as defined in \cite{G2}, to each 
representation $\pi_j$, one can associate a unipotent orbit
${\mathcal O}_{G_j}(\pi_j)$. Similarly, we define the set ${\mathcal
O}_{G_l}(E_\tau(\cdot,s))$. See \cite{G2} section 5. Thus, for $1\le
j\le l-1$, we may associate with each representation $\pi_j$, a unipotent subgroup of $G_j$, which we shall
denote by $V_j(\pi_j)$, and a character $\psi_{V_j(\pi_j)}$ of
$V_j(\pi_j)(F)\backslash V_j(\pi_j)({\bf A})$ such that the Fourier
coefficient
$$L_{\pi_i}(g_i)=\int\limits_{V_i(\pi_i)(F)\backslash V_i(\pi_i)({\bf
A})}\varphi_{\pi_i}(v_ig_i)\psi_{V_i(\pi_i)}(v_i)dv_i$$  is not zero
for some choice of data. Similarly, for the Eisenstein series $E_\tau(\cdot,s)$we can associate an integral $L_\tau$ which is defined in a similar way as above.

Suppose that for $\text{Re} (s)$ large, after unfolding integral
\eqref{global1}, we obtain the integral
\begin{equation}\label{global2}
\int\limits_{Z({\bf A})M({\bf A})\backslash G({\bf
A})}L^{R_1}_{\pi_1}(g)L^{R_2}_{\pi_2}(g)\cdots
L^{R_{l-1}}_{\pi_{l-1}}(g)f_{L^{R_l}_\tau}(w_0g)dg
\end{equation}
Here, $M$ is a certain subgroup of $G$, and
$$L^{R_i}_{\pi_i}(g)=\int\limits_{R_i({\bf
A})}L_{\pi_i}(rg)\psi_{R_i}(r)dr$$ where $R_i$ is a certain
unipotent subgroup of $G_i$, and $\psi_{R_i}$ is an additive
character defined on this group.

To make things clear we give an example. Consider the integral which represents the exterior square $L$ function defined on a cuspidal representation $\pi_1$ of $GL_4({\bf A})$. This integral was introduced in \cite{J-S}. Here $G=GL_2$, and integral \eqref{global1} is given by
\begin{equation}\label{ext1}
\int\limits_{Z({\bf A})GL_2(F)\backslash GL_2({\bf A})}\varphi_{\pi_1}^{U_1,\psi_{U_1}}(g) E^{U_2,\psi_{U_2}}(g,s)dg\notag
\end{equation}
Here $U_1$ is the subgroup of $GL_4$ consisting of all unipotent matrices of the form $\begin{pmatrix} I_2&X\\ &I_2\end{pmatrix}$, and $\psi_{U_1}$ is the character given by $\psi(\text{tr}X)$. Also, $E(g,s)$ is the standard Eisenstein series defined on the group $GL_2({\bf A})$, and $U_2$ is the trivial group. Unfolding this integral one obtains 
\begin{equation}\label{ext2}
\int\limits_{Z({\bf A})N({\bf A})\backslash GL_2({\bf A})}
L_{\pi_1}^{R_1}(g)f(g,s)dg\notag
\end{equation}
where
$$L_{\pi_1}^{R_1}(g_1)=\int_{\bf A}W_{\pi_1}(k(r_1)g_1)dr_1$$ Here $W_{\pi_1}$ is the Whittaker coefficient of $\pi_1$, and $k(r_1)=I_4+r_1e_{2,3}$ where $e_{2,3}$ is the matrix of size four which has one at the $(2,3)$ entry and zero elsewhere.

\begin{definition}\label{def1}
In the above notations, if, for $\text{Re}(s)$ large, after an
unfolding process, integral \eqref{global1} is equal to integral
\eqref{global2}, then  we refer to integral \eqref{global1} as to a
unipotent global integral. The number $l$ is the length of the
integral, and if all the functionals $L^{R_i}_{\pi_i}$ and
$L^{R_l}_\tau$ are factorizable, then we say that integral
\eqref{global1} is an Eulerian unipotent integral.
\end{definition}

In \cite{G1} we give a general overview about the motivation for the
above construction and definition. We also motivate some of the
discussion below.  As explained in \cite{G1} and also in \cite{G3},
we are mainly interested in Eulerian unipotent integrals which
satisfies the dimension equation
\begin{equation}\label{dim1}
\sum_{i=1}^{l}(\text{dim}\ \pi_i - \text{dim}\ U({\mathcal
O}_{G_i}))\  = \text{dim}\ G - \text{dim}\ Z
\end{equation}
where we write $\pi_l$ for the Eisenstein series $E_\tau$. Our goal
is to study the following
\begin{problem}\label{prob1}
Classify all Eulerian unipotent integrals which satisfies the
dimension equation \eqref{dim1}.
\end{problem}
In this paper we study the above problem for the group $G=GL_m$.
Moreover, we consider those cases where the center $Z$ of $GL_m$
coincides with the center of each of the groups $G_i$. This
restricts the choice of the  groups $G_i$, and also the relevant
unipotent orbits ${\mathcal O}_{G_i}$. We classify these cases in
the next section.

Thus, for these cases we study the following
\begin{conjecture}\label{conj1}
Assume that $G=GL_m$. Suppose that the integral \eqref{global1} is a
nonzero Eulerian unipotent integral which satisfies the dimension
equation \eqref{dim1}. Assume that the representations $\sigma_i$ generated by the Fourier coefficient
$\varphi_{\pi_i}^{U_i,\psi_{U_i}}(g)$ are not a one dimensional representation of $G({\bf A})$.  Then the length $l$ of the integral \eqref{global1}  is less than or equal to three.
\end{conjecture}

It is easy to construct examples with $l=3$. The Rankin product
integral given by
$$\int\limits_{Z({\bf A})G(F)\backslash G({\bf
A})}\varphi_{\pi_1}(g)\varphi_{\pi_2}(g)E(g,s)dg$$ is an Eulerian
unipotent integral. Here $\pi_i$ are irreducible cuspidal
representations of $G({\bf A})$, and $E(g,s)$ is a suitable
Eisenstein series defined on this group. We also mention that in
\cite{G3} there is a classification of all integrals as above, when
now $\pi_2$ is an arbitrary automorphic representation, and $E(g,s)$
is an arbitrary degenerate Eisenstein series.

Notice that since the group $G=GL_m$ contains a nontrivial unipotent subgroup , then
$\text{dim}\ \pi_i > \text{dim}\ U({\mathcal O}_{G_i})$. This follows from our assumption about the groups $\sigma_i$. Indeed,
since we assume that $\pi_i$ has a nonzero Fourier coefficient
corresponding to the unipotent orbit ${\mathcal O}_{G_i}$, then
clearly $\text{dim}\ \pi_i \ge \text{dim}\ U({\mathcal O}_{G_i})$.
If there is an equality, then we have $ {\mathcal
O}_{G_i}(\pi_i)={\mathcal O}_{G_i}$.  This means that the representation $\sigma_i$ is a one
dimensional representation of the group $G({\bf A})$. To see this, let
$x_\alpha(r)$ denote the one dimensional unipotent subgroup of
$G=GL_m$ which corresponds to some simple root of this group. Then
we consider the integral
$$\int\limits_{F\backslash {\bf A}}\varphi_{\pi_i}^{U_i,\psi_{U_i}}
(x_\alpha(r)g)\psi(r)dr$$ It is not hard to show that this Fourier
coefficient corresponds to a unipotent orbit which is strictly
greater than ${\mathcal O}_{G_i}$. In section 7 we will prove this
when $m=2$. The general case is similar. So, if we assume that $
{\mathcal O}_{G_i}(\pi_i)={\mathcal O}_{G_i}$, then we conclude that
the above integral must be zero for all choice of data. But then, if
we consider the group $SL_2$ which is generated by $x_{\pm
\alpha}$, we deduce that
$\varphi_{\pi_i}^{U_i,\psi_{U_i}}(mg)=\varphi_{\pi_i}^{U_i,\psi_{U_i}}(g)$
for all $m\in SL_2({\bf A})$. This implies that as a representation of
$GL_m({\bf A})$, the representation $\sigma_i$
is a one dimensional representation. Since, see conjecture
\ref{conj1}, we assume that this is not the case, we deduce that
$\text{dim}\ \pi_i > \text{dim}\ U({\mathcal O}_{G_i})$.

Notice that this proves that conjecture \ref{conj1} hold when $G$ is
a group of type $A_1$. Indeed, in this case $\text{dim}\ G -
\text{dim}\ Z= 3$ and hence $l\le 3$.

A few general remarks. First, consider the case when  $l=1$. Since
we want the global integral to depend on a complex variable  then,
excluding the simple cases of Hecke type integrals, we assume that
the only representation appearing in integral \eqref{global1}, is
an Eisenstein series. Since we require that the integral converges,
then  we must take $G$ to be trivial. Hence integral \eqref{global1}
reduces to the Fourier coefficient $E_\tau^{U_l,\psi_{U_l}}(g_l,s)$.
These type of integrals were studied in various references, see
\cite{S}, and are known as the Langlands- Shahidi type integrals.
Henceforth, we consider integrals as integral \eqref{global1} such
that $l\ge 2$.

As a second remark, we  mention that the integrals of the form
\eqref{global1} are not the only known Eulerian unipotent integrals.
An extension of these integrals is when the unipotent groups are
defined by a diagonal embedding. For example it is not hard to show
that for a suitable Eisenstein series on $GL_4({\bf A})$, the
integral
\begin{equation}\label{example35}
\int\varphi_{\sigma}(h)
E_{\tau_1}\left ( \begin{pmatrix} I&X\\ &I\end{pmatrix}\begin{pmatrix} h&\\
&h\end{pmatrix},s_1\right )E_{\tau_2}\left ( \begin{pmatrix} I&X\\ &I\end{pmatrix}\begin{pmatrix} h&\\
&h\end{pmatrix},s_2\right ) \psi(\text{tr}X)dXdh\notag
\end{equation}
is an Eulerian unipotent integral. Here $\sigma$ is an irreducible
cuspidal representation of $GL_2({\bf A})$.  This issue is discussed
in  \cite{G1}.

\section{\bf The Relevant Unipotent Orbits}

In this section we classify the relevant unipotent orbits whose
stabilizer is the group $GL_m$. We only consider those unipotent
orbits such that the center $Z$ of $GL_m$ coincides with the center
of $G_i$. We describe the orbits, the corresponding unipotent
groups, and their characters. We also collect some information
needed later.

{\bf 1)}\ Consider the group $GL_{km}$ with $k,m\ge 2$, and denote
${\mathcal O}=(k^m)$. The corresponding unipotent groups are
$$U_{k,m}({\mathcal O})=\left \{ \begin{pmatrix} I&X_1&*&*&*\\
&I&X_2&*&*\\ &&I&*&*\\
&&&\ddots&X_{k-1}\\ &&&&I\end{pmatrix},\ \ X_j\in Mat_{m\times m}
\right \}$$ The character is given by
$$\psi_{U_{k,m}({\mathcal O})}(u)=\psi(\text{tr}(X_1+X_2+\cdots
+X_{k-1}))$$ The stabilizer of $\psi_{U_{k,m}({\mathcal O})}$ inside
the group $GL_m\times\cdots\times GL_m$, counted $k$ times, is the
diagonal embedding of $GL_m$ which we denote by $GL_m^\Delta$.
Notice that the center of the stabilizer coincide with the center of
$GL_{km}$.

For the classical groups, it follows from \cite{C-M} page 88 , and
for the exceptional groups it follows from \cite{C} page 400 , that
if $m\ge 4$, then the above cases are the only cases where the
stabilizer is $GL_m$ whose center $Z$ is the center of the group
$G_i$.

{\bf 2)}\ In case when $m=3$, beside the cases discussed in {\bf
1)}\ there is another case in the exceptional groups. Let $GE_6$
denote the similitude group of the exceptional group $E_6$. We can
realize this group as a subgroup of $E_8$. We will use the notations
from \cite{G4}. It follows from \cite{C} page 402 that the unipotent
orbit whose label is $D_4$, its stabilizer inside a suitable Levi
subgroup is the group of type $A_2$. In \cite{G4} page 106, one
defines the unipotent group $U$ with character $\psi_U$ whose
stabilizer is the group $GL_3$. As explained in that reference the
embedding of $GL_3$ is such that its center $Z$ is the center of
$GE_6$.

{\bf 3)}\ Finally we consider the case of $m=2$, that is the  group
$GL_2$. Beside the cases in {\bf 1)}\ we have other cases  in other
groups. Let $GSp_{2(2n+1)}$ denote the similitude group of the
symplectic group $Sp_{2(2n+1)}$. In matrices we take the symplectic
form $\begin{pmatrix} &J_{2n+1}\\ -J_{2n+1}& \end{pmatrix}$ where
$J_{2n+1}$ is the $2n+1$ matrix with ones on the other diagonal and
zeros elsewhere. Let ${\mathcal O}=((2n+1)^2)$, and let
$U_n({\mathcal O})$ denote the standard unipotent radical subgroup
of the parabolic subgroup of $GSp_{2(2n+1)}$, whose Levi part is
$GL_2\times\cdots\times GL_2$ counted $n$ times. We take
$U_n({\mathcal O})$ to consist of upper unipotent matrices. Identify
the quotient $U_n({\mathcal O})/[U_n({\mathcal O}),U_n({\mathcal
O})]$ with $L=Mat_2\oplus\cdots\oplus Mat_2$, counted $n$ times.
Then the group $GL_2\times\cdots\times GL_2$ acts on $U_n({\mathcal
O})/[U_n({\mathcal O}),U_n({\mathcal O})]$ by conjugation, and
there is an open orbit for this action. For $(X_1, X_2,\ldots ,
X_n)\in L$ define the character $\psi(\text{tr}\ (X_1+X_2+\cdots
+X_n))$, and extend it trivially to a character $\psi_{U_n({\mathcal
O})}$ of $U_n({\mathcal O})(F)\backslash U_n({\mathcal O})({\bf
A})$. The stabilizer of $\psi_{U_n({\mathcal O})}$ inside
$GL_2\times\cdots\times GL_2$ is $GL_2$ embedded diagonally. Its not
hard to check that the center of $GL_2$ coincides with the center of
$GSp_{2(2n+1)}$.

A similar situation occurs in the similitude group $GSO_{4n}$. Let
${\mathcal O}=((2n)^2)$. Let $U_n({\mathcal O})$ denote the standard
unipotent radical of the parabolic subgroup whose Levi part is
$GL_2\times\cdots\times GL_2$ counted $n$ times. Identify
$U_n({\mathcal O})/[U_n({\mathcal O}),U_n({\mathcal O})]$ with
$L=Mat_2\oplus\cdots\oplus Mat_2\oplus Mat_2^0$, where $Mat_2$
appears $n$ times and $Mat_2^0=\{ \begin{pmatrix} y&\\
&-y\end{pmatrix}\ : y\in {\bf A}\}$. Then the group
$GL_2\times\cdots\times GL_2$ acts on $U_n({\mathcal
O})/[U_n({\mathcal O}),U_n({\mathcal O})]$ by conjugation, and
there is an open orbit for this action. For $(X_1, X_2,\ldots ,
X_{n-1},y)\in L$ define the character $\psi(\text{tr}\
(X_1+X_2+\cdots +X_{n-1})+y)$, and extend it trivially to a
character $\psi_{U_n({\mathcal O})}$ of $U_n({\mathcal
O})(F)\backslash U_n({\mathcal O})({\bf A})$. The stabilizer of
$\psi_{U_n({\mathcal O})}$ inside $GL_2\times\cdots\times GL_2$ is
$GL_2$ embedded diagonally.

Finally, consider the similitude group $GE_7$. We use the notations
as in \cite{G4}. In the notations of \cite{C}, let ${\mathcal
O}=E_6$. Let $U({\mathcal O})$ denote the unipotent radical subgroup
of the parabolic subgroup whose Levi part contains the group
$GL_2\times GL_2\times GL_2$ embedded in $GE_7$ as follows. In terms of the roots of $E_7$, the
group $GL_2\times GL_2\times GL_2$ contains the group $SL_2\times
SL_2\times SL_2$ generated by $x_{\pm 0100000};\ x_{\pm 0000100};\
x_{\pm 0000001}$. Thus $\text{dim}\ U({\mathcal O})=60$. Define a
character $\psi_{U({\mathcal O})}$ of the group $U({\mathcal
O})(F)\backslash U({\mathcal O})({\bf A})$ as follows. For $u\in
U({\mathcal O})$ write
$$u=x_{1000000}(r_1)x_{0010000}(r_2)x_{0101000}(r_3)x_{0001100}(r_4)
x_{0000110}(r_5)x_{0000011}(r_6)u'$$ Here $u'$ is an element of
$U({\mathcal O})$ which is a product of one parameter unipotent
subgroups none of which are among the above six roots. Define
$\psi_{U({\mathcal O})}(u)=\psi(r_1+r_2+\cdots +r_6)$. Then the
stabilizer of $\psi_{U({\mathcal O})}$ inside $GL_2\times GL_2\times
GL_2$ is the diagonal group $GL_2$. It follows from \cite{G4} that
the center of this copy of $GL_2$ is the center of $GE_7$.

\section{\bf On Some Global  Integrals}

In this section we study global integrals of the type
\eqref{global1} which we assume to satisfy  the dimension equation
\eqref{dim1}.  The result we obtain in this section is a set of
integrals which satisfies the dimension equation. We emphasize that
the global integrals we obtain  are not necessarily nonzero or that
they are an Eulerian unipotent integrals. We consider these issues
for some cases, in the next sections.

We recall from section 2 that we assume that the representation
$\pi_1$ is an irreducible cuspidal representation of the group
$G_1({\bf A})$, and that the representation $\pi_l$ is an Eisenstein
series. From section 3 we deduce that the group $G_1$ is one of the following groups. First,  $G_1=GL_{k_1m}$ with $k_1\ge
1$ and $m\ge 2$. In addition, if $m=3$, then we can also have
$G_1=GE_6$, and if $m=2$, then we also have that $G_1$ is one of the
groups $GSp_{2(n+1)}, GSO_{4n}$ or $GE_7$. In the first case we will
write $k$ for $k_1$. Also, we denote $G=GL_m$, and by $Z$ the center
of $G$. In the first subsection we deal with the group $G=GL_2$,
then with $G=GL_3$ and in the last subsection we consider the group
$G=GL_m$ for $m\ge 4$.

\subsection{\bf The case when $G=GL_2$}

Assume that $G=GL_2$. The dimension equation is then
\begin{equation}
\sum_{i=1}^{l}(\text{dim}\ \pi_i - \text{dim}\ U({\mathcal
O}_{G_i}))\ = 3\notag
\end{equation}
It follows from the dimension equation that we have $l=2,3$. We will
use the notation $\lambda_H$ for the partition $\lambda$ of the
classical group $H$.

In  Table 1 we consider all integrals \eqref{global1} with $l=2$.
There are two cases to consider. In the first case we have
$\text{dim}\ \pi_1 - \text{dim}\ U({\mathcal O}_{G_1})=2$ and
$\text{dim}\ E_\tau - \text{dim}\ U({\mathcal O}_{G_2})=1$, and in
the second case we have  $\text{dim}\ \pi_1 - \text{dim}\
U({\mathcal O}_{G_1})=1$ and $\text{dim}\ E_\tau - \text{dim}\
U({\mathcal O}_{G_2})=2$. The first case is listed in the second
column of Table 1, and the second case in the third column of Table
1.

Consider the first case. If $\pi_1$ is defined on $GL_{2k}({\bf
A})$, then we have
$$2=\text{dim}\ \pi_1-\text{dim}\ U({\mathcal O}_{G_1})=
\frac{1}{2}(\text{dim}\ {\mathcal O}_{GL}(\pi_1)-\text{dim} (k^2)_{GL})$$ The
only partition that satisfies this equation is ${\mathcal O}_{GL}(\pi_1)=((k+2)(k-2))_{GL}$.
But $\pi_1$ is cuspidal, and hence generic. The only cuspidal
representation $\pi_1$ such that ${\mathcal
O}_{GL}(\pi_1)=((k+2)(k-2))_{GL}$ is when   $k=2$, and we obtain
${\mathcal O}_{GL}(\pi_1)=(4)_{GL}$. As another example from this
case, assume that $\pi_1$ is defined on $GSp_{2(2n+1)}({\bf A})$. In
this case the only partition which satisfies
$$2=\text{dim}\ \pi_1-\text{dim}\ U({\mathcal O}_{G_1})=
\frac{1}{2}(\text{dim}\ {\mathcal O}_{GSp}(\pi_1)-\text{dim} ((2n+1)^2)_{GSp})$$
is the partition ${\mathcal O}_{GSp}(\pi_1)=((2n+4)(2n-2))_{GSp}$. In contrast to the case
when $\pi_1$ was defined on $GL_{2k}$, in this case cuspidal
representations of $GSp_{2(2n+1)}({\bf A})$ which satisfy ${\mathcal
O}(\pi_1)=((2n+4)(2n-2))_{GSp}$ do exist. These are some CAP
representations, which can be constructed, for example, using the
method described in \cite{G5}. The other two cases, which appear in the second column of Table 1 are constructed in a similar way. 

As for the Eisenstein series, the situation is similar. In this case we have $\text{dim}\ E_\tau -\text{dim}\ U({\mathcal O}_{G_2})=1$. If, for example, $E_\tau$ is defined on the group $GSO_{4p}({\bf A})$, then the only partition that satisfies this last equation is $((2p+1)(2p-1))$. This explains the last entry in the second column of Table 1.

\begin{table}
\begin{center}\begin{tabular}{ | l | r | r | r |}
\hline         & $(4)_{GL}$ \ \ \ \ \ \ \ \ \ \  &   $(2)_{GL}$ \ \ \ \ \ \ \ \ \ \   \\
       ${\mathcal O}(\pi_1)$ & $((2n+4)(2n-2))_{GSp}$  & $((2n+2)(2n))_{GSp}$   \\
                        & $((2n+3)(2n-3))_{GSO}$   &   $((2n+1)(2n-1))_{GSO}$  \\
                        &   $E_7(a_1)$ \ \ \ \ \ \ \ \ \ \          &   $E_7(a_2)$ \ \ \ \ \ \ \ \ \ \  \\    \hline
                        &  $((p+1)(p-1))_{GL}$  &  $((p+2)(p-2))_{GL}$ \\
${\mathcal O}(E_\tau)$  &  $((2p+2)(2p))_{GSp}$  &  $((2p+4)(2p-2))_{GSp}$\\
                        &  $((2p+1)(2p-1))_{GSO}$  &  $((2p+3)(2p-3))_{GSO}$\\
                        &     $E_7(a_2)$ \ \ \ \ \ \ \ \ \ \ &   $E_7(a_1)$\ \ \ \ \ \ \ \ \ \ \\ \hline

\end{tabular}
\bigskip  \caption{} \label{ta1}
\end{center}
\end{table}

\begin{table}
\begin{center}\begin{tabular}{ | l | r | r | r |}
\hline       &$(2)_{GL}$ \ \ \ \ \ \ \ \ \ \ \\
${\mathcal O}(\pi_1)$ & $((2n+2)(2n))_{GSp}$  \\
                      & $((2n+1)(2n-1))_{GSO}$  \\
                      & $E_7(a_2)$ \ \ \ \ \ \ \ \ \ \  \\ \hline
                      &  $((q+1)(q-1))_{GL}$ \\
${\mathcal O}(\pi_2)$  &  $((2q+2)(2q))_{GSp}$ \\
                       &  $((2q+1)(2q-1))_{GSO}$ \\
                       &     $E_7(a_2)$ \ \ \ \ \ \ \ \ \ \ \\ \hline
                       &  $((p+1)(p-1))_{GL}$ \\
${\mathcal O}(E_\tau)$  &  $((2p+2)(2p))_{GSp}$ \\
                       &  $((2p+1)(2p-1))_{GSO}$ \\
                       &     $E_7(a_2)$ \ \ \ \ \ \ \ \ \ \ \\
                       \hline

\end{tabular}
\bigskip  \caption{} \label{ta2}
\end{center}
\end{table}

When $l=3$ we have
$$\text{dim}\ \pi_1-\text{dim}\ U({\mathcal O}_{G_1})=
\text{dim}\ \pi_2-\text{dim}\ U({\mathcal O}_{G_2})= \text{dim}\
E_\tau-\text{dim}\ U({\mathcal O}_{G_3})=1$$ In Table 2 we list all
relevant possible cases.

As an example, consider the case when all three representations are
defined on the exceptional group $E_7({\bf A})$. Since, see \cite{C-M},  $E_7(a_2)$ is
the only unipotent orbit which satisfies
$$1=\text{dim}\ \pi_1-\text{dim}\ U({\mathcal O}_{G_1})=
\frac{1}{2}(\text{dim}\ {\mathcal O}_{GE_7}(\pi_1)-\text{dim} E_6)$$ this
explains the relevant entry in the second table. The entries for
$\pi_2$ and $E_\tau$ are the same. Thus, for a corresponding integral to
exist, we need to prove that there is a cuspidal representation
$\pi_1$, and an Eisenstein series $E_\tau$, both defined on
$GE_7({\bf A})$ such that ${\mathcal O}(\pi_1)={\mathcal O}(E_\tau)=
E_7(a_2)$.

We emphasize that the entries in each table are independent. For
example, in Table 2 there are $4^3=64$ cases to consider. This means that using the data from Table 2, we can construct 64 type of integrals, as defined by integral \eqref{global1}. Clearly, we still need to check which of these integral is
well defined, and which is a nonzero Eulerian unipotent integral.

\subsection{\bf The case when $G=GL_3$}

Assume that $G=GL_3$. In this section all partitions are partitions
of the group $GL$. It follows from section 3 that $\pi_1$ is defined
on $GL_{3k}$ for some $k\ge 1$, or defined on the group $GE_6({\bf
A})$.

In this case, the dimension equation is given by
\begin{equation}
\sum_{i=1}^{l}(\text{dim}\ \pi_i - \text{dim}\ U({\mathcal
O}_{G_i}))\ = 8\notag
\end{equation}
First, we claim that for any irreducible representation $\pi$ of the
group $H=GL_{3k}({\bf A})$ such that ${\mathcal O}(\pi)> (k^3)$, or
for the group $H=GE_6({\bf A})$ such that ${\mathcal O}(\pi)> E_6$, we have
$\text{dim}\ \pi - \text{dim}\ U({\mathcal O}_{H}))\ge 2$. Indeed,
if $H=GL_{3k}$, then  $\text{dim}\ \pi - \text{dim}\ U({\mathcal
O}_{H})=\frac{1}{2}( {\mathcal O}(\pi)- \text{dim}\ (k^3))$. The first
partition which strictly greater than $(k^3)$  is $((k+1)k(k-1))$.
Hence $\text{dim}\ \pi - \text{dim}\ U({\mathcal O}_{H})\ge
\frac{1}{2}(\text{dim}\ (k+1)k(k-1))-\text{dim}\ (k^3))=2$, and the
claim follows. In the case of $H=GE_6$, we note that the first
partition which is strictly greater than $D_4$ is $D_5(a_1)$, and it
satisfies $\frac{1}{2}(\text{dim}\ D_5(a_1) - \text{dim}\ D_4)=2$.
See \cite{C-M}.

We need the following  result which we shall prove in the last section.
\begin{lemma}\label{lem1}
There are no irreducible nonzero cuspidal representations $\pi$ defined on
$GE_6({\bf A})$, such that ${\mathcal O}_{GE_6}(\pi)$ is equal to
$D_5$, or to $D_5(a_1)$.
\end{lemma}

Next we claim that if $\pi$ is an irreducible cuspidal representation of
$H=GL_{3k}({\bf A})$ or of the group $H=GE_6({\bf A})$ such that
${\mathcal O}(\pi)> E_6$, then we have $\text{dim}\ \pi -
\text{dim}\ U({\mathcal O}_{H})\ge 3$. For the exceptional group
this claim follows from lemma \ref{lem1}. For the group $GL_{3k}$,
it follows from the fact that every cuspidal representation is
generic.

Returning to the global integral \eqref{global1} which satisfies the
dimension equation, we assume first that $\pi_1$ is a cuspidal
representation of the group $GL_{3k}({\bf A})$. It follows from
proposition \ref{prop2}, which we state and prove in the next
subsection, that $k=1,2$. Assume first that $k=2$. Then $\pi_1$ is
defined on $GL_6({\bf A})$, and hence ${\mathcal O}(\pi_1)=(6)$.
Hence $\text{dim}\ \pi_1 - \text{dim}\ U({\mathcal O}_{G_1})=6$.
Thus, we deduce that $l=2$, and $\text{dim}\ E_\tau - \text{dim}\
U({\mathcal O}_{G_2})=2$. From this we obtain Table 3.

Next consider the case when $k=1$. This means that $\pi_1$ is a
cuspidal representation of $GL_3({\bf A})$. Hence its dimension is
3, and we obtain
\begin{equation}
\sum_{i=2}^{l}(\text{dim}\ \pi_i - \text{dim}\ U({\mathcal
O}_{G_i}))\ = 5\notag
\end{equation}
Since we proved that for any representation $\pi$ as above
$\text{dim}\ \pi - \text{dim}\ U({\mathcal O}_{H})\ge 2$ then
$l=2,3$. If $l=2$, then $\text{dim}\ E_\tau - \text{dim}\
U({\mathcal O}_{G_2})=5$. We then obtain Table 4.

Assume that $l=3$. Thus $\text{dim}\ E_\tau - \text{dim}\
U({\mathcal O}_{G_3})=2,3$. Assume first that it is equal to three.
Then $E_\tau$ is defined either on $GL_{3p}({\bf A})$ for some $p\ge
1$, and satisfies ${\mathcal O}(E_\tau)=((p+1)^2(p-2)),\
((p+2)(p-1)^2)$, or $E_\tau$ is defined on $GE_6({\bf A})$ and
satisfies ${\mathcal O}(E_\tau)=E_6(a_3)$. Since $l=3$, then we have
a third representation, denoted by $\pi_2$ which satisfies
$\text{dim}\ \pi_2 - \text{dim}\ U({\mathcal O}_{G_2})=2$. Hence,
the representation is defined either on $GL_{3q}({\bf A})$ for some
$q\ge 1$, and satisfies ${\mathcal O}(\pi_2)=((q+1)q(q-1))$ or
defined on $GE_6({\bf A})$ and satisfies ${\mathcal
O}(\pi_2)=D_5(a_1)$. From this we obtain the second column of Table
5.

\begin{table}
\begin{center}\begin{tabular}{ | l | r | r | r |}
\hline ${\mathcal O}(\pi_1)$ & $(6)$\ \ \ \ \ \ \ \ \ \    \\
\hline  & $((p+1)p(p-1))$   \\
               ${\mathcal O}(E_\tau)$ & $D_5(a_1)$ \ \ \ \ \ \ \  \\
                                 \hline
\end{tabular}
\bigskip  \caption{} \label{ta3}
\end{center}
\end{table}

\begin{table}
\begin{center}\begin{tabular}{ | l | r | r | r |}
\hline ${\mathcal O}(\pi_1)$ &   $(3)$\ \ \ \ \ \ \ \ \ \ \ \ \ \     \\
\hline  &   $((p+2)(p+1)(p-3))$  \\
               ${\mathcal O}(E_\tau)$ &   $((p+3)(p-1)(p-2))$  \\
                               &    $E_6(a_1)$\ \ \ \ \ \ \ \ \ \ \ \ \\  \hline
\end{tabular}
\bigskip  \caption{} \label{ta4}
\end{center}
\end{table}

\begin{table}
\begin{center}\begin{tabular}{ | l | r | r | r |}
\hline ${\mathcal O}(\pi_1)$ &  $(3)$\ \ \ \ \ \ \ \ \ &$(3)$\ \ \ \ \ \ \ \ \  \\
 \hline                        & $((q+1)q(q-1))$  &  $((q+1)^2(q-2))$ \\
  ${\mathcal O}(\pi_2)$    &  $D_5(a_1)$\ \ \ \ \ \ \      &   $((q+2)(q-1)^2)$ \\
                             &                &  $E_6(a_3)$\ \ \ \ \ \ \ \   \\
\hline                 & $((p+1)^2(p-2))$  &  $((p+1)p(p-1))$ \\
          ${\mathcal O}(E_\tau)$  &  $((p+2)(p-1)^2)$    &  $D_5(a_1)$\ \ \ \ \ \   \\
                              &  $E_6(a_3)$\ \ \ \ \ \ \  &\\ \hline

\end{tabular}
\bigskip  \caption{} \label{ta5}
\end{center}
\end{table}
The last column in Table 5  is obtained by interchanging the roles
of the representations $\pi_2$ and $E_\tau$. Thus, the third column
corresponds to the case when $\text{dim}\ E_\tau - \text{dim}\
U({\mathcal O}_{G_3})=2$.

The final case to consider in this subsection is when $\pi_1$
defines an irreducible cuspidal representation of $GE_6({\bf A})$.
Since $\varphi_{\pi_1}^{U_1,\psi_{U_1}}(g)$ is not zero, it follows
from section 3 that ${\mathcal O}_{GE_6}(\pi_1)>D_4$. Hence, the
possibilities for ${\mathcal O}_{GE_6}(\pi_1)$ are  $E_6, E_6(a_1),
D_5, E_6(a_3)$ or $D_5(a_1)$. However, the representation $\pi_1$ is
cuspidal. Hence, from lemma \ref{lem1} it follows that ${\mathcal
O}_{GE_6}(\pi_1)=E_6, E_6(a_1)$ or $E_6(a_3)$.

In all cases we have that $\text{dim}\ U({\mathcal O}_{G_1})=30$.
Hence, it follows from \cite{C-M} page 129 that $\text{dim}\ \pi_1 -
\text{dim}\ U({\mathcal O}_{G_1})= 6, 5, 3$. Thus, we have the
corresponding three possibilities
\begin{equation}
\sum_{i=2}^{l}(\text{dim}\ \pi_i - \text{dim}\ U({\mathcal
O}_{G_i})) =2, 3, 5\notag
\end{equation}
From the fact that $\text{dim}\ \pi_i - \text{dim}\ U({\mathcal
O}_{G_i})\ge 2$ we deduce that when ${\mathcal O}_{GE_6}(\pi_1)=E_6$
then we have $l=2$ and in the other two cases we have $l=2,3$. We
summarize all possible cases in Tables 6 and 7.
\begin{table}
\begin{center}\begin{tabular}{ | l | r | r | r |}
\hline ${\mathcal O}(\pi_1)$ & $E_6$\ \ \ \ \ \ \ \ \ \  &  $E_6(a_1)$\ \ \ \ \ \ \  & $E_6(a_3)$\ \ \ \ \ \ \ \ \ \    \\
\hline  & $((p+1)p(p-1))$ &  $((p+1)^2(p-2))$ &  $((p+2)(p+1)(p-3))$ \\
               ${\mathcal O}(E_\tau)$ & $D_5(a_1)$\ \ \ \ \ \  &  $((p+2)(p-1)^2)$ & $((p+3)(p-1)(p-2))$ \\
                               &   &  $E_6(a_3)$\ \ \ \ \ \  & $E_6(a_1)$\ \ \ \ \ \ \ \ \ \ \\  \hline
\end{tabular}
\bigskip  \caption{} \label{ta6}
\end{center}
\end{table}

\begin{table}

\begin{center}\begin{tabular}{ | l | r | r | r |}
\hline ${\mathcal O}(\pi_1)$ &  $E_6(a_3)$\ \ \ \ \ \ &$E_6(a_3)$\ \ \ \ \ \ \\
 \hline                        & $((q+1)q(q-1))$  &  $((q+1)^2(q-2))$ \\
  ${\mathcal O}(\pi_2)$    &  $D_5(a_1)$\ \ \ \ \  \   &   $((q+2)(q-1)^2)$ \\
                             &                &  $E_6(a_3)$\ \ \ \ \ \   \\
\hline                 & $((p+1)^2(p-2))$  &  $((p+1)p(p-1))$ \\
          ${\mathcal O}(E_\tau)$  &  $((p+2)(p-1)^2)$    &  $D_5(a_1)$\ \ \ \ \ \  \\

                              &  $E_6(a_3)$\ \ \ \ \ \ &\\ \hline

\end{tabular}
\bigskip  \caption{} \label{ta7}
\end{center}
\end{table}
For example, the last column in  Table 6 corresponds to the case
when $l=2$ and ${\mathcal O}(\pi_1)=E_6(a_3)$. In this case we have
$\text{dim}\ \pi_1 - \text{dim}\ U({\mathcal O}_{G_1})=3$, and since
$l=2$, we deduce that $\text{dim}\ E_\tau - \text{dim}\ U({\mathcal
O}_{G_2})=5$. If $E_\tau$ is defined on $GL_{3p}({\bf A})$ for some
$p\ge 1$, then the only options are ${\mathcal
O}(E_\tau)=((p+2)(p+1)(p-3))$ or ${\mathcal
O}(E_\tau)=((p+3)(p-1)(p-2))$. If $E_\tau$ is defined on $GE_6({\bf
A})$, then the only possibility is ${\mathcal O}(E_\tau)=E_6(a_1)$.

This completes the case when $G=GL_3$. Notice that we proved

\begin{proposition}\label{prop1}
Given a global integral of the form \eqref{global1}, where $G=GL_3$,
which satisfies the dimension equation \eqref{dim1}, then $l\le 3$.
\end{proposition}

\subsection{\bf The case when $G=GL_m$ with $m\ge 4$}

Let $G=GL_m$ with $m\ge 4$. It follows from section 3 that  we may
assume that every automorphic representation $\pi_i$ which appears
in integral \eqref{global1} is defined on $GL_{k_im}({\bf A})$ for
some $k_i\ge 1$.   As before, we assume that $\pi_1$ is an
irreducible cuspidal representation of $GL_{km}$ where we write $k$
for $k_1$.

The following proposition is valid for all $m\ge 2$.
\begin{proposition}\label{prop2}
Suppose that  $m\ge 2$.  Then $k=1, 2$.

\end{proposition}

\begin{proof}
From the dimension equation \eqref{dim1} we obtain
\begin{equation}
\sum_{i=2}^{l}(\text{dim}\ \pi_i - \text{dim}\ U({\mathcal
O}_{G_i}))\ +\text{dim}\ \pi_1 - \text{dim}\ U({\mathcal O}_{G_1}) =
\text{dim}\ G - \text{dim}\ Z\notag
\end{equation}
Since $\pi_1$ is cuspidal, then it is generic and $\text{dim}\
\pi_1=\frac{1}{2}\text{dim}\ (km)= \frac{1}{2}km(km-1)$. We also
have $\text{dim}\ U_{k,m}({\mathcal O})=\frac{1}{2}k(k-1)m^2$. Hence
the dimension equation is
\begin{equation}
\sum_{i=2}^{l}(\text{dim}\ \pi_i - \text{dim}\ U({\mathcal
O}_{G_i}))\ + \frac{1}{2}km(km-1)- \frac{1}{2}k(k-1)m^2=m^2-1\notag
\end{equation}
This is the same as
\begin{equation}
\sum_{i=2}^{l}(\text{dim}\ \pi_i - \text{dim}\ U({\mathcal
O}_{G_i}))\ + (\frac{1}{2}km-m-1)(m-1)=0\notag
\end{equation}
If $k\ge 3$ then the left hand side is a positive number. Hence, we
must have $k=1,2$.

\end{proof}

Assume first that $k=2$. Then the above equation becomes
\begin{equation}\label{dim2}
\sum_{i=2}^{l}(\text{dim}\ \pi_i - \text{dim}\ U({\mathcal
O}_{G_i}))=m-1
\end{equation}
Next consider the Fourier coefficient $E_\tau^{U_l,\psi_{U_l}}(g,s)$
which appears in integral \eqref{global1}. This Eisenstein series is
defined on the group $G_l({\bf A})$. Assume  that $G_l=GL_{pm}$ for
some $p\ge 1$. The unipotent orbit attached to the Fourier
coefficient of the Eisenstein series is $(p^m)$, hence ${\mathcal
O}_{GL_{pm}}(E_\tau)> (p^m)$, or ${\mathcal O}_{GL_{pm}}(E_\tau)\ge
((p+1)p^{m-2}(p-1))$. Thus, from the formula for the dimension of a
partition, see \cite{C-M}, we obtain
$$\text{dim}\ E_\tau - \text{dim}\ U_{p,m}({\mathcal
O})=\frac{1}{2}(\text{dim}\ {\mathcal
O}_{GL_{pm}}(E_\tau)-\text{dim} (p^m))\ge$$ $$\ge
\frac{1}{2}(\text{dim}\ ((p+1)p^{m-2}(p-1))- \text{dim}\
(p^m))=m-1$$ Combining this with equation \eqref{dim2}, we deduce
that when $k=2$, we also have $l=2$.

Thus, the case when $\pi_1$ is an irreducible cuspidal
representation of $GL_{2m}({\bf A})$ produces Table 8.
\begin{table}
\begin{center}\begin{tabular}{ | l | r | r | r |}
\hline ${\mathcal O}(\pi_1)$ & $(2m)$\ \ \ \ \ \ \ \ \ \     \\
\hline  ${\mathcal O}(E_\tau)$ & $((p+1)p^{m-2}(p-1))$\\   \hline
\end{tabular}
\bigskip  \caption{} \label{ta8}
\end{center}
\end{table}

Assume that $k=1$. In this case $\pi_1$ is a cuspidal representation
of $GL_m({\bf A})$. Assuming that integral \eqref{global1} satisfies
the dimension equation, does not by itself limit the possibilities
as in the previous cases. In some more details, it follows from the
proof of proposition \ref{prop2} that the dimension equation is
\begin{equation}
\sum_{i=2}^{l}(\text{dim}\ \pi_i - \text{dim}\ U({\mathcal
O}_{G_i}))\ =(\frac{1}{2}m+1)(m-1)\notag
\end{equation}
It is not hard to produce examples of representations which
satisfies this equation. For example, when $m=4$, the right hand
side of the above equation is equal to 9. It is not hard to
construct an Eisenstein series $E_\tau$ on $GL_4({\bf A})$ such that
$\text{dim}\ E_\tau=3$. Indeed, let $\tau$ be the trivial representation and assume that $E_\tau$ is the Eisenstein series associated with the induced representation $Ind_{P({\bf A})}^{GL_4({\bf A})}\delta_P^s$. Here $P$ is the maximal parabolic subgroup of $GL_4$ whose Levi part is $GL_3\times GL_1$.  Hence, the integral
$$\int\limits_{Z({\bf A})GL_4(F)\backslash GL_4({\bf
A})}\varphi_{\pi_1}(g)E_\tau(g,s_1)E_\tau(g,s_2)E_\tau(g,s_3)dg$$
satisfies the dimension equation. Notice that this does not
necessarily  mean that conjecture \ref{conj1} is not true for $m\ge
4$, since in that conjecture we assume that the global integral is
nonzero. And indeed this is what happens in the above integral. A
simple unfolding implies that it is identically zero.

\section{\bf On Some Eisenstein Series}

In this section we study some Eisenstein series needed to construct
integrals of the type \eqref{global1}. More precisely, the tables
produced in the previous section assumes the existence of certain
Eisenstein series with certain Fourier coefficients. In this section
we indicate how to construct such Eisenstein series.

Given a reductive classical group $H$, it follows from \cite{C-M} that  unipotent orbits of $H$ are parameterized by certain partitions. Given such a partition $\lambda$, we emphasize the dependence on $H$ by writing $\lambda_H$ instead of $\lambda$. Given two
partitions ${\frak b}_1=(k_1k_2\ldots k_p)$ and ${\frak
b}_2=(m_1m_2\ldots m_q)$ of the numbers $n$ and $r$, we set ${\frak
b}_1+{\frak b}_2=((k_1+m_1)(k_2+m_2)\ldots )$. We also write
$2{\frak b}={\frak b}+{\frak b}$.

Let $H$ denote a reductive group, and let $P=MU$ denote a maximal
parabolic subgroup of $H$. Let $\tau$ denote an automorphic
representation of $M({\bf A})$, and denote by $E_\tau(\cdot,s)$ the
Eisenstein series associated with the induced representation
$Ind_{P({\bf A})}^{H({\bf A})}\tau\delta_P^s$. Notice that by
induction of stages, this covers all possible Eisenstein series. We
are interested in the set ${\mathcal O}_H(E_\tau(\cdot,s))$ for
$\text{Re}(s)$ large. By that we mean in the domain where the
Eisenstein series is defined by a convergent series. For the
classical groups we have the following

\begin{proposition}\label{prop3}
With the above notations, for $\text{Re}(s)$ large, we have

{\bf 1)}\ For $H=GL_n$, assume that $M=GL_a\times GL_{n-a}$, and
$\tau=\tau_1\otimes \tau_2$. Then we have
$${\mathcal O}_{GL_n}(E_\tau(\cdot,s))={\mathcal O}_{GL_a}(\tau_1)
+{\mathcal O}_{GL_{n-a}}(\tau_2)$$

{\bf 2)}\ For $H=GSp_{2n}$, assume that $M=GL_a\times GSp_{2(n-a)}$,
and $\tau=\tau_1\otimes \tau_2$. Then we have
$${\mathcal O}_{GSp_{2n}}(E_\tau(\cdot,s))=2{\mathcal O}_{GL_a}(\tau_1)
+{\mathcal O}_{GSp_{2(n-a)}}(\tau_2)$$

{\bf 3)}\ For $H=GSO_{2n}$, assume that $M=GL_a\times GSO_{2(n-a)}$,
and $\tau=\tau_1\otimes \tau_2$. Then we have
$${\mathcal O}_{GSO_{2n}}(E_\tau(\cdot,s))=2{\mathcal O}_{GL_a}(\tau_1)
+{\mathcal O}_{GSO_{2(n-a)}}(\tau_2)$$

In particular, if $H$ is one of the above classical groups, then
$\text{dim}\ {\mathcal O}_{H}(E_\tau(\cdot,s))=\text{dim}\ \tau
+\text{dim}\ U$.

\end{proposition}

\begin{proof}
The proof of this proposition is a straight forward computation of
the relevant unipotent orbit, which is done by unfolding the
Eisenstein series. The computations in general, are very similar to
the computations done in \cite{G3} proposition 1. We omit the
details.

The last equation in the statement of the proposition follows from
the fact that from the above parts {\bf 1)-3)}, we deduce that
${\mathcal O}_{H}(E_\tau(\cdot,s))$ is the induced orbit as defined
in \cite{C-M} section 7. Then the statement about the dimension
follows from this reference lemma 7.2.5.

\end{proof}

Since we are mainly interested in the case when $m=2$, we work out
the relevant Eisenstein series in this case only. In other words,
the Eisenstein series which appear in tables \ref{ta1} and
\ref{ta2}. From the above proposition we deduce,

\begin{lemma}\label{lem2}
{\bf A)}\ For the group $GL_{2p}$ we have the following cases,

{\bf 1)}\ Suppose that ${\mathcal O}(E_\tau(\cdot,s))=((p+1)(p-1))$.
Then, there is  $0\le i\le 2$ such that $M=GL_{2(a-1)+i}\times
GL_{2(p-a+1)-i}$ ; ${\mathcal O}(\tau_1)=(a(a-2+i))$ and ${\mathcal
O}(\tau_2)=((p-a+1)(p-a+1-i))$.

{\bf 2)}\ Suppose that ${\mathcal O}(E_\tau(\cdot,s))=((p+2)(p-2))$.
Then, there is $0\le i\le 4$ such that $M=GL_{2(a-2)+i}\times
GL_{2(p-a+2)-i}$ ; ${\mathcal O}(\tau_1)=(a(a-4+i))$ and ${\mathcal
O}(\tau_2)=((p-a+2)(p-a+2-i))$.

{\bf B)}\ For the group $GSp_{4p+2}$ we have the following cases,

{\bf 3)}\ Suppose that ${\mathcal O}(E_\tau(\cdot,s))=((2p+2)(2p))$.
Then, there is  $0\le i\le 1$ such that $M=GL_{2a-i}\times
GSp_{2(2p-2a+i+1)}$ ; ${\mathcal O}(\tau_1)=(a(a-i))$ and ${\mathcal
O}(\tau_2)=((2p-2a+2)(2p-2a+2i))$.

{\bf 4)}\ Suppose that ${\mathcal
O}(E_\tau(\cdot,s))=((2p+4)(2p-2))$. Then, there is  $0\le i\le 3$
such that $M=GL_{2a-i}\times GSp_{2(2p-2a+i+1)}$ ; ${\mathcal
O}(\tau_1)=(a(a-i))$ and ${\mathcal
O}(\tau_2)=((2p-2a+4)(2p-2a+2i-2))$.

{\bf C)}\ For the group $GSO_{4p}$ we have the following cases,

{\bf 5)}\ Suppose that ${\mathcal
O}(E_\tau(\cdot,s))=((2p+1)(2p-1))$. Then, there is  $0\le i\le 1$
such that $M=GL_{2a-i}\times GSO_{2(2p-2a+i)}$ ; ${\mathcal
O}(\tau_1)=(a(a-i))$ and ${\mathcal
O}(\tau_2)=((2p-2a+1)(2p-2a+2i-1))$.

{\bf 6)}\ Suppose that ${\mathcal
O}(E_\tau(\cdot,s))=((2p+3)(2p-3))$. Then, there is  $0\le i\le 3$
such that $M=GL_{2a-i}\times GSO_{2(2p-2a+i)}$ ; ${\mathcal
O}(\tau_1)=(a(a-i))$ and ${\mathcal
O}(\tau_2)=((2p-2a+3)(2p-2a+2i-3))$.

\end{lemma}

\begin{proof}
The proof follows immediately from Proposition \ref{prop3}. We give some details about the first case. 

Assume that ${\mathcal
O}(E_\tau(\cdot,s))=((p+1)(p-1))$.
Assume that ${\mathcal O}(\tau_1)=(\alpha_1\beta_1)$ and that ${\mathcal O}(\tau_2)=(\alpha_1\beta_2)$ then it follows from Proposition \ref{prop3} that $\alpha_1+\alpha_2=p+1$ and $\beta_1+\beta_2=p-1$. Assume that $\tau_1$ is an automorphic representation of $GL_{2(a-1)+i}({\bf A})$ for some $a$ and $i=1,2$. Then $\tau_2$ is an automorphic representation of $GL_{2(p-a+1)-i}({\bf A})$. This means that $\alpha_1+\beta_1=2a+i-2$ and that $\alpha_2+\beta_2=2p-2a-i+2$. Also, we have $\alpha_1\ge \beta_1$ and $\alpha_2\ge \beta_2$. From these six relations the claim follows. Indeed, we obtain the relations $\alpha_1=p+1-\alpha_2$, then $\beta_1=2a+i-p+\alpha_2-3$ and $\beta_2=2p-2a-i-\alpha_2+2$. The relation $\alpha_1\ge \beta_1$ implies $2p+4\ge 2a+i+2\alpha_2$, and the second inequality implies $2a+i+2\alpha_2\ge 2p+2$. Hence, $2a+i+2\alpha_2=2p+2, 2p+3, 2p+4$. If $i=1$ then we must have $2a+i+2\alpha_2=2p+3$ from which it follows that $\alpha_1=a$ . From this the claim follow
 s. Similar result happens when $i=2$. We omit the details.

\end{proof}

For the exceptional groups we proceed as follows. We use the
following lemma which is a version of proposition 5.16 in \cite{G2}.

\begin{lemma}\label{lemex}
Let $H$ denote an exceptional group, and let $E_\tau(\cdot,s)$
denote an Eisenstein series attached to $Ind_{P({\bf A})}^{H({\bf
A})}\tau\delta_P^s$. Here $P$ is a maximal parabolic subgroup of $H$
with Levi decomposition $P=MU$. Let $\tau$ denote an automorphic
representation of $M({\bf A})$. Then, for $\text{Re}(s)$ large, we
have $\text{dim}\ {\mathcal O}_{H}(E_\tau(\cdot,s))=\text{dim}\ \tau
+\text{dim}\ U$.
\end{lemma}

\section{\bf On some Fourier Expansions}

Let $\pi$ denote an automorphic representation of the group $H({\bf
A})$, where $H$ is one of the groups $GL_{2k},\ GSp_{2(2k+1)},\
GSO_{4k}$, or $GE_7$. Let $V$ be any one of the unipotent subgroups
defined in section 3 part {\bf 1)} with $m=2$, or part {\bf 3)}. Let
$\psi_V$ denote the character of $V(F)\backslash V({\bf A})$ defined
in that section in each case. Then the stabilizer of $\psi_V$
contains the group $GL_2$.

For $g\in GL_2$, define
$$f(g)=\int\limits_{V(F)\backslash V({\bf A})}\varphi_\pi(vg)\psi_{V}(v)dv$$
In this section we compute the following integral
\begin{equation}\label{whgl1}
\int\limits_{F\backslash {\bf A}}f\left (\begin{pmatrix} 1&y\\
&1\end{pmatrix}g\right )\psi(ay)dy
\end{equation}
where $a=0,1$.

\begin{lemma}\label{fourgl1}
{\bf a)}\ When $a=1$, integral \eqref{whgl1} corresponds to the
Fourier coefficient of $\pi$ associated with the unipotent orbit\\
{\bf 1)}\ $((k+1)(k-1))$ if $H=GL_{2k}$.\\
{\bf 2)}\ $((2k+2)(2k))$ if $H=GSp_{2(2k+1)}$.\\
{\bf 3)}\ $((2k+1)(2k-1))$ if $H=GSO_{4k}$.\\
{\bf 4)}\ $E_7(a_2)$ if $H=GE_7$.

{\bf b)}\ When $a=0$, the constant term of $f(g)$ corresponds to a
sum of Fourier coefficients associated with  every unipotent orbit
of $H$ which is strictly greater than the unipotent orbit written in
part {\bf a)},  and a certain Fourier coefficient which contains the
constant term  specified below, as an inner integration.
\end{lemma}

\begin{proof}
We work out the details for the case where $H=GL_{2k}$. The other
cases are similar. We compute the integral
$$\int\limits_{F\backslash {\bf A}}
\int\limits_{V(F)\backslash V({\bf A})}\varphi_\pi(v
\mu(y))\psi_{V}(v)\psi(ay)dydv$$ where
$$\mu(y)=I_{2k}+ye_{1,2}+ye_{3,4}+\cdots +ye_{2k-1,2k}$$ Here, we
denote by $e_{i,j}$ the matrix of order $2k$ with one at the $(i,j)$
entry and zero elsewhere. Expand the above integral along the
subgroup
$$V_1=\{v_1(r_2,r_3,\ldots, r_k)=
I_{2n}+r_2e_{3,4}+r_3e_{5,6}+\cdots +r_ke_{2k-1,2k}\ \ :\ \ r_i\in
{\bf A}\}$$ Then the above integral is equal to
$$\int\sum_{\xi_i\in F}\int \varphi_\pi(v_1(r_2,r_3,\ldots, r_k)v \mu(y))\psi_{V}(v)
\psi(ay+\sum_{i=2}^k\xi_i r_i)dydvdr_i$$ Let
$$L_1=\{l_1(z_2,z_3,\ldots,z_k)=I_{2k}+z_2e_{2,3}+z_3e_{4,5}+\cdots
+z_ke_{2k-2,2k-1}\}$$ Then $L_1$ is a subgroup of $V$. Since
$\varphi_\pi$ is an automorphic function, then we have
$\varphi_\pi(m)=\varphi_\pi(l_1(\xi_2,\xi_3,\ldots,\xi_k)m)$.
Conjugating this discrete matrix from left to right, collapsing
summation with integration, the last integral is equal to
$$\int\limits_{L_1 ({\bf A})}\ \
\int\limits_{V_2(F)\backslash V_2({\bf A})}\varphi_\pi(v
l_1)\psi_{a,V_2}(v)dvdl_1$$ Here
$$V_2=\{ u\in U\ :\ u_{i,i+1}=0;\ i=2,4,\ldots, 2k-2\}$$ where $U$
is the maximal standard unipotent subgroup of $GL_{2k}$. Also, for
$v\in V_2$ we have
$$\psi_{a,V_2}(v)=\psi(av_{1,2}+v_{1,3}+v_{2,4}+v_{3,5}+\cdots
+v_{2k-2,2k})$$

Assume that $a=1$. Then, using the correspondence between unipotent orbits and Fourier coefficients as described in
\cite{G2} section 2, we deduce that the integration over $V_2$ in the above
integral, is a Fourier coefficient associated with the unipotent
orbit $((k+1)(k-1))$.

Next, assume $a=0$. For this section let $U$ denote the standard maximal unipotent subgroup of $GL_{2k}$. Let $w$ denote the Weyl element of $GL_{2k}$
defined as follows. For all $1\le i\le k$ set
$w_{i,2i-1}=w_{k+i,2i}=1$, and all other entries of $w$ are zeros.
Since $w$ is a discrete element, then $\varphi_\pi$ is left
invariant by it.  Conjugating $w$ from left to right, the above
integral is equal to
$$\int\limits_{L_1 ({\bf A})}\  \int\limits_{L_2(F)\backslash
L_2({\bf A})}\ \int\limits_{V_3(F)\backslash V_3({\bf
A})}\varphi_\pi(vl_2w l_1)\psi_{V_3}(v)dvdl_2dl_1$$ Here $V_3$ is
the subgroup of $U$ defined by
$$V_3=\{ v\in U\ :\ u_{i,j}=0;\ 2\le i\le k\  \text{and}\ k\le j\le
2k-1\}$$ and $L_2$ is the group of all lower unipotent matrices
defined by
$$L_2=\{ l_2=\begin{pmatrix} I_k& \\ X&I_k\end{pmatrix}\ :\ X_{i,j}=0;
\ j\le i+1\}$$ Also, the character $\psi_{V_3}$ is defined as
follows
\begin{equation}\label{char1}
\psi_{V_3}(v)=\psi(v_{1,2}+v_{2,3}+\cdots +v_{k-1,k}+v_{k+1,k+2}+
v_{k+2,k+3}+\cdots +v_{2k-1,2k})
\end{equation}
Let $L_3$ denote the subgroup of $U$ defined by
$$L_3=\{l_3=\begin{pmatrix} I_k&Y \\ &I_k\end{pmatrix}\ :\
Y_{i,j}=0;\ i\le j\ \text{and}\ i=k\}$$ Next we expand the above
integral along the group $L_3(F)\backslash L_3({\bf A})$. Carrying
out a similar process as in the previous expansion in this section,
and doing it one variable at the time, we obtain that the above
integral is equal to
\begin{equation}\label{four10}
\int\limits_{L_1 ({\bf A})}\  \int\limits_{L_2 ({\bf A})}\
\int\limits_{V_4(F)\backslash V_4({\bf A})}\varphi_\pi(vl_2w
l_1)\psi_{V_4}(v)dvdl_2dl_1
\end{equation}
Here $V_4$ is the subgroup of $U$ defined by
$$V_4=\{ u\in U\ :\ u_{k,j}=0;\ k\le j\le 2k-1\}$$ The character
$\psi_{V_4}$ is the trivial extension of $\psi_{V_3}$ to
$\psi_{V_4}$.

Next, consider the expansion of the above integral along the one
parameter unipotent subgroup of $U$ consisting of all matrices of
the form $x_\alpha(r)=I_{2k}+re_{k,2k-1}$. First consider the
contribution from the nontrivial orbit. It is a sum over $\xi\in
F^*$  of  the Fourier coefficients
$$\int\limits_{L_1 ({\bf A})}\  \int\limits_{L_2 ({\bf A})}\
\int\limits_{F\backslash {\bf A}}\int\limits_{V_4(F)\backslash
V_4({\bf A})}\varphi_\pi(x_\alpha(r)vl_2w l_1)\psi_{V_4}(v)\psi(\xi
r)drdvdl_2dl_1$$ Using the corresponding between unipotent orbits
and Fourier coefficients, as described in \cite{G2} section 2, we
deduce that the above Fourier coefficient corresponds to the
unipotent orbit $((k+2)(k-2))$. The second contribution to integral
\eqref{four10} from the above expansion is from the constant term,
and it is equal to
$$\int\limits_{L_1 ({\bf A})}\  \int\limits_{L_2 ({\bf A})}\
\int\limits_{V_5(F)\backslash V_5({\bf A})}\varphi_\pi(vl_2w
l_1)\psi_{V_5}(v)dvdl_2dl_1$$ where
$$V_5=\{ u\in U\ :\ u_{k,j}=0;\ k\le j\le 2k-2\}$$ We further expand
this integral along $x_\alpha(r)=I_{2k}+re_{k,2k-2}$. Again we get
two contributions. The first, from the nontrivial orbit, contributes
a sum of Fourier coefficient, each corresponds to the unipotent
orbit $((k+3)(k-3))$. The second is the constant term. Arguing by
induction we eventually expand along the group
$x_\alpha(r)=I_{2k}+re_{k,k+1}$. The contribution from the
nontrivial orbit will produce a sum of Fourier coefficient which
corresponds to the unipotent orbit $(2k)$, and the trivial orbit
will produce the integral
$$\int\limits_{L_1 ({\bf A})}\  \int\limits_{L_2 ({\bf A})}\
\int\limits_{U(F)\backslash U({\bf A})}\varphi_\pi(ul_2w
l_1)\psi_{U}(v)dudl_2dl_1$$ where $\psi_U$ is the character defined
by \eqref{char1} extended trivially to $U$. Notice that this last
integral contains the constant term of $\pi$ along the unipotent
radical of the maximal parabolic group whose Levi part is
$GL_k\times GL_k$. To summarize, we expressed integral \eqref{whgl1} as a sum of Fourier coefficients which corresponds to unipotent orbits which are greater than $(k+1)(k-1))$, and an integral containing a constant term as an inner integration. This is the statement in part {\bf b)} of the lemma.

We finish the proof of the lemma with the description of the
constant term which is obtained in the other cases. First, in the
classical groups. In the case when $H=GSp_{2(2k+1)}$ we obtain the
constant term along the unipotent radical of the maximal parabolic
subgroup whose Levi part is $GL_{2k+1}$. When $H=GSO_{4k}$ we get
the unipotent radical of the maximal parabolic subgroup whose Levi
subgroup is $GL_{2k}$. Finally, when $H=GE_7$ we obtain the
unipotent radical of the maximal parabolic subgroup whose Levi part
is $E_6$.

\end{proof}

\section{\bf Unfolding Global Integrals with $G=GL_2$}

It follows from Tables 1 and 2 that there are two cases to consider
for the group $G=GL_2$. First, if $l=2$ the global integral
\eqref{global1} is
\begin{equation}\label{gl21}
\int\limits_{Z({\bf A})G(F)\backslash G({\bf
A})}\varphi_{\pi_1}^{U_1,\psi_{U_1}}(g)E_\tau^{U_2,\psi_{U_2}}(g,s)dg
\end{equation}
and second, if $l=3$ we have,
\begin{equation}\label{gl22}
\int\limits_{Z({\bf A})G(F)\backslash G({\bf
A})}\varphi_{\pi_1}^{U_1,\psi_{U_1}}(g)\varphi_{\pi_2}^{U_2,\psi_{U_2}}(g)
E_\tau^{U_3,\psi_{U_3}}(g,s)dg
\end{equation}
For $1\le j\le 3$, let $G_j$ denote one of the groups $GL_{2p},\
GSp_{2(2p+1)},\ GSO_{4p}$ or $GE_7$. In integrals \eqref{gl21} and
\eqref{gl22}, we assume that $\pi_1$ is a cuspidal representation.
For $j=1,2$, the sets ${\mathcal O}_{G_j}(\pi_j)$ are listed in
Tables 1 and 2, and similarly for $j=2,3$ for ${\mathcal
O}_{G_j}(E_\tau)$.

In this section we determine which of the above integrals, assuming
that the representations involved in it satisfy the requirements of
Tables 1 and 2,  is a nonzero global unipotent integral. To do that
we carry out the unfolding process. We start with the unfolding of
the Eisenstein series. Let $V=U_2$ or $V=U_3$ be one of the
unipotent groups introduced in section 3 with $m=2$. Thus, for the
group $H=GL_{2p}$, then $V=U_{p,2}({\mathcal O})$, and for the other
classical groups we let $V=U_n({\mathcal O})$. For $H=GE_7$ this
group was denoted in section 3 by $U({\mathcal O})$. By $\psi_V$ we
denote the corresponding character which was defined in section 3.
Assume that $E_\tau(\cdot,s)$ is associated with the induced
representation $Ind_{P({\bf A})}^{H({\bf A})}\tau\delta_P^s$. Here
$H=G_2$ when we consider integral \eqref{gl21}, and  $H=G_3$ when we consider integral \eqref{gl22}. We also assume that $P$ is a maximal
parabolic subgroup of $H$. The representation $\tau$ is an
automorphic representation of $M({\bf A})$ where $M$ is the Levi
part of $P$. Denote by $U(P)$ the unipotent radical of $P$. We 
denote by $U$ the unipotent maximal subgroup of $GL_{2p}$ consisting of
upper triangular matrices. We have
\begin{equation}\label{four1}
E_\tau^{V,\psi_V}(h,s)= \int\limits_{V(F)\backslash V({\bf
A})}\sum_{\gamma\in P(F)\backslash H(F)}f_\tau(\gamma
vh,s)\psi_V(v)dv=
\end{equation}
$$=\sum_{\gamma\in P(F)\backslash
H(F)/V(F)}\ \int\limits_{V^\gamma(F)\backslash V({\bf
A})}f_\tau(\gamma vh,s)\psi_V(v)dv$$ where $V^\gamma=V\cap
\gamma^{-1}P\gamma$.

We call an element $\gamma\in P(F)\backslash H(F)/V(F)$  not
admissible, if there exists an element $v\in V$ such that $\gamma
v\gamma^{-1}\in U(P)$, and such that $\psi_V(v)\ne 1$. Otherwise we
say that $\gamma$ is admissible.  From the definition it follows
that a non-admissible element contributes zero to the above
summation. Our goal is to characterize all the admissible elements.
We shall write all details for the group $H=GL_{2p}$. For the other
classical groups and for $GE_7$ the process is similar.

It follows from the Bruhuat decomposition that every element in
$P(F)\backslash H(F)/V(F)$ can be written as $\gamma=wv_w$. Here $w$
is a Weyl element of $GL_{2p}$, and
$v_w=I_{2p}+z_1e_{1,2}+z_2e_{3,4}+\cdots + z_pe_{2p-1,2p}$ where
$z_i\in F$ and $e_{i,j}$ is the matrix of size $2p$ with one at the
$(i,j)$ entry, and zero elsewhere. We claim that if $w$ is not
admissible, then $wv_w$ is also not admissible. This follows from
the action by conjugation of $v_w$ on the group $V$. Indeed, if
there is a $v\in V$ such that $wvw^{-1}\in U(P)$, and $\psi_V(v)\ne
1$, then we can find an element $v'\in V$ such that $\psi_V(v')\ne
1$, and that $v_wv'v_w^{-1}=v$. From this the claim follows.

Assume that the Levi part of $P$ is $GL_r\times GL_{2p-r}$ with
$p\le r$. Then $2p-r\le r$. Assume that $w$ is admissible. We shall
write $w[i,j]$ for its $(i,j)$ entry. Thus $w[i,j]=0,1$. By a
suitable multiplication from the left by an element of $GL_r\times
GL_{2p-r}$, we may assume that there is a maximal number $q'$ such
that $0\le q'\le 2p-r$ and such that $w[r+i,i]=1$ for all $1\le i\le
q'$. From the maximality of $q'$ we obtain that $w[j,q'+1]=0$ for
all $r+1\le j\le 2p$. Hence, adjusting by an element of $GL_r\times
GL_{2p-r}$, we may assume that $w[1,q'+1]=1$. It is convenient to
consider the cases $q'$ even or odd separately. Assume that $q'=2q$.
Then $w[1,2q+1]=1$. Consider $v=I_{2p}+ze_{2q+1,2q+3}$. Then
$\psi_V(v)\ne 1$. Consider the matrix $wvw^{-1}$. A simple matrix
multiplication implies that if $w[j,2q+3]=1$ for some $r+1\le j\le
2p$, then $wvw^{-1}\in U(P)$ and hence $w$ is not admissible. Thus,
we have $w[j',2q+3]=1$ for some $2\le j'\le r$. By a suitable
multiplication from the left by elements in $GL_r\times GL_{2p-r}$, we may assume that
$w[2,2q+3]=1$. The process is inductive, namely using the same
argument we deduce that $w[3,2q+5]=1$, and so on until
$w[p-q,2p-1]=1$. Thus we have determined the first $p-q$ rows of
$w$.

Consider the next $r-p+q$ rows. Multiplication from the left by elements in  $GL_r\times
GL_{2p-r}$, we let $q_0$ be the smallest positive number such that
$w[p-q+1,2(q+q_0)]=1$. Then arguing as above we deduce that $w$ is
admissible if and only if, after a suitable multiplication by
$GL_r\times GL_{2p-r}$, we have $w[p-q+2,2(q+q_0+1)]=1$, and so on.
Let $q_1$ be such that $w[r,2(q+q_1)]=1$. Using the same argument,
$w$ is admissible if and only if  $2(q+q_1)=2p$. In other words, $w$
is admissible if and only if 
$$w[r,2p]=w[r-1,2(p-1)]=\ldots =
w[p-q+1,2(2p-r-q+1)]=1$$
Thus, so far we determined the first $r+2q$
rows of $w$. But, up to multiplication by $GL_r\times GL_{2p-r}$,
this determines also the last $2p-2q-r$ rows. In other words we have
$w[r+2q+1,2q+2]=w[r+2q+2,2q+4]=\ldots =w[2p, 2(2p-q-r)]=1$.

A similar construction holds when $q'=2q+1$. Writing $q$ for $q'$,
we can parameterize the admissible Weyl elements by the elements
$w_{q}$ with $0\le q\le 2p-r$, where
$$w_{q}=\begin{pmatrix} &L_{q}'&\\ &&I_{2(r-p)+q}\\ I_q&&\\
&L_q''&\end{pmatrix}$$ Here $I_k$ is the identity matrix of size
$k$, and
$$L_q'=\{x\in \text{Mat}_{(2p-r-q)\times 2(2p-r-q)}\ :\ x_{i,2i-1}=1\ \text{for}\ 1\le
i\le 2p-r-q\ \text{and}\ x_{i,j}=0\ \text{elsewhere}\}$$
$$L_q''=\{y\in \text{Mat}_{(2p-r-q)\times 2(2p-r-q)}\ :\ y_{i,2i}=1\
\text{for}\ 1\le i\le 2p-r-q\ \text{and}\ y_{i,j}=0\
\text{elsewhere}\}$$ Notice that
$$w_{2p-r}=\begin{pmatrix} &I_r\\ I_{2p-r}&\end{pmatrix}$$
From all this we conclude that we may consider those double cosets
whose representatives are of the form $w_qz(r_1,r_2,\ldots,r_p)$
where $r_i\in F$ and
$$z(r_1,r_2,\ldots,r_p)=I_{2p}+r_1e_{1,2}+r_2e_{3,4}+\cdots
+r_pe_{2p-1,2p}$$ To eliminate more double cosets, we specify the
Eisenstein series as in lemma \ref{lem2} part {\bf A}. Thus we
assume that ${\mathcal O}(E_\tau(\cdot,s))$ is equal to
$((p+1)(p-1))$ or $((p+2)(p-2))$. Also, for reasons which will be
clear later, we further restrict the Eisenstein series, and assume
that $i$, as appears in lemma \ref{lem2}, is odd.  In other words,
we assume that, using induction by stages, there is a number $a$ and
an odd number $i$ such that $E_\tau(\cdot,s)$ is induced from that
parabolic subgroup. Thus, we have three cases. First, if ${\mathcal
O}(E_\tau(\cdot,s))=((p+1)(p-1))$ then we assume that
$M=GL_{2a-1}\times GL_{2(p-a)+1}$ and ${\mathcal
O}(\tau_1)=(a(a-1))$ and ${\mathcal O}(\tau_2)=((p-a+1)(p-a))$. In
the second case ${\mathcal O}(E_\tau(\cdot,s))=((p+2)(p-2))$ and
there are two possible induction data. The first possibility is
$M=GL_{2a-3}\times GL_{2(p-a)+3}$ and ${\mathcal
O}(\tau_1)=(a(a-3))$ and ${\mathcal O}(\tau_2)=((p-a+2)(p-a+1))$.
The second possibility is $M=GL_{2a-1}\times GL_{2(p-a)+1}$ and
${\mathcal O}(\tau_1)=(a(a-1))$ and ${\mathcal
O}(\tau_2)=((p-a+2)(p-a-1))$. However,  changing $a$ in the first
possibility with $p-a+2$ gives us the second possibility.

To summarize, if ${\mathcal O}(E_\tau(\cdot,s))=((p+1)(p-1))$ then
the induction data is ${\mathcal O}(\tau_1)=(a(a-1))$ and ${\mathcal
O}(\tau_2)=((p-a+1)(p-a))$. If ${\mathcal
O}(E_\tau(\cdot,s))=((p+2)(p-2))$, then the induction data is
${\mathcal O}(\tau_1)=(a(a-1))$ and ${\mathcal
O}(\tau_2)=((p-a+2)(p-a-1))$.

We return to the computation of $E_\tau^{V,\psi_V}(h,s)$. Suppose
that $V$ contains a subgroup $V_1$ such that
$\gamma^{-1}V_1\gamma\subset M$. Then the right most integral in
identity \eqref{four1} contains the integral
$$\int\limits_{V_1(F)\backslash V_1({\bf
A})}f_\tau(\gamma^{-1}v_1\gamma h')\psi_V(v_1)dv_1$$ as inner
integration. This integral defines certain Fourier coefficients of
the automorphic functions $\varphi_{\tau_1}$ and $\varphi_{\tau_2}$.
If, for some $\gamma$ the unipotent orbit corresponding to one of
these Fourier coefficients is strictly greater than ${\mathcal
O}(\tau_1)$ or ${\mathcal O}(\tau_2)$, then the above integral is
zero, and hence the contribution to \eqref{four1} from this
representative is zero.

Let $\gamma=w_qz(r_1,r_2,\ldots,r_p)$. To handle the elements
$z(r_1,r_2,\ldots,r_p)$, we consider the subgroup $V'$ of $V$
defined by $V'=\{ v\in V\ :\ v_{2j,2j-1}=0;\ 1\le j\le p-1\}$.
Notice that $z(r_1,r_2,\ldots,r_p)$ normalizes $V'$, and if by
restriction, we consider the character  $\psi_V$ as a character of
$V'$, then $\psi_V(
z(r_1,r_2,\ldots,r_p)^{-1}v'z(r_1,r_2,\ldots,r_p))=\psi_V(v')$.
Therefore, if we replace $V$ by $V'$ and take $V_1$ to be a subgroup
of $V'$, then we may ignore the unipotent part of $\gamma$. Recall
that $i$, as defined in lemma \ref{lem2} is odd. This means that
both numbers $r$ and $2p-r$ are odd. We recall that $GL_r\times GL_{2p-r}$ is the Levi part of $P$, the parabolic subgroup we used to construct  the Eisenstein series.
The Weyl elements which we
still need to consider are given by $w_q$ where we take $q=2t,2t+1$
with $0\le t\le \frac{2p-r-1}{2}$. It follows from matrix
multiplication that after conjugating by $w_{2t}$ and by $w_{2t+1}$
we obtain on $\varphi_{\tau_1}$ the Fourier coefficient
corresponding to the unipotent orbit $((p-t)(r-p+t))$, and on
$\varphi_{\tau_2}$ the Fourier coefficient corresponding to the
unipotent orbit $((2p-r-t)t)$.

Because of the induction data of the Eisenstein series, as given in
lemma \ref{lem2}, these unipotent orbits must satisfy
$((2p-r-t)t)\le (a(a-1))$ or $((p-t)(r-p+t))\le (a(a-1))$. For
otherwise, the above integral will be zero. But $(a(a-1))$ is the
smallest unipotent orbit of $GL_{2p}$ of the form $(n_1n_2)$, and
hence either $((2p-r-t)t)=(a(a-1))$ or $((p-t)(r-p+t))=(a(a-1))$. In
both cases we obtain $t=\frac{2p-r-1}{2}$. Thus, in the
factorization of \eqref{four1}  we are left with two possible
nonzero contributions corresponding to the Weyl elements
$w_{2p-r-1}$ and $w_{2p-r}$.

To continue we now unfold the global integrals \eqref{gl21} and
\eqref{gl22}.  First, unfolding the Eisenstein series, we notice
that we need to consider representatives of the space of double
cosets $P\backslash H/V\cdot G$. Using the above discussion, we only
need to consider two types of representatives. They are
$w_{2p-r-1}z(r_1,r_2,\ldots,r_p)$ and
$w_{2p-r}z(r_1,r_2,\ldots,r_p)$. However, it is not hard to check
that all of these representatives, which were distinct when we
considered $P\backslash H/V$, are now reduced to one element, which
we choose to be $w_{2p-r-1}$. We denote this element by $w_0$.

Thus, for $\text{Re}(s)$ large, integral \eqref{gl21} is equal to
\begin{equation}\label{gl23}
\int\limits_{Z({\bf A})B(F)\backslash G({\bf
A})}\varphi_{\pi_1}^{U_1,\psi_{U_1}}(g) \int\limits_{V^{w_0}({\bf
A})\backslash V({\bf
A})}f_\tau^{V^{w_0},\psi_{w_0}}(w_0vg,s)\psi_V(v)dvdg
\end{equation}
and integral \eqref{gl22} is equal to
\begin{equation}\label{gl24}
\int\limits_{Z({\bf A})B(F)\backslash G({\bf
A})}\varphi_{\pi_1}^{U_1,\psi_{U_1}}(g)\varphi_{\pi_2}^{U_2,\psi_{U_2}}(g)
\int\limits_{V^{w_0}({\bf A})\backslash V({\bf
A})}f_\tau^{V^{w_0},\psi_{w_0}}(w_0vg,s)\psi_V(v)dvdg
\end{equation}
Here $B$ is the Borel subgroup of $G=GL_2$ which consists of all
upper unipotent matrices, and
$$f_\tau^{V^{w_0},\psi_{w_0}}(h,s)=
\int\limits_{V^{w_0}(F)\backslash V^{w_0}({\bf
A})}f_\tau(v_0h,s)\psi_{w_0}(v_0)dv_0$$  If the Levi part of $P$ is
$GL_{2a-1}\times GL_{2(p-a)+1}$, then this Fourier coefficient
corresponds to the unipotent orbit $(a(a-1))$ of $GL_{2a-1}$, and
corresponds to the unipotent orbit $((p-a+1)(p-a))$ of
$GL_{2(p-a)+1}$.

Let $U(B)$ denote the unipotent radical of the Borel group $B$.
Consider first the case when ${\mathcal
O}(E_\tau(\cdot,s))=((p+1)(p-1))$. Then the induction data is
${\mathcal O}(\tau_1)=(a(a-1))$ and ${\mathcal
O}(\tau_2)=((p-a+1)(p-a))$. Arguing in a similar way as in the proof
of the first part of lemma \ref{fourgl1}, we deduce that for all
$u\in U(B)({\bf A})$ we have
$f_\tau^{V^{w_0},\psi_{w_0}}(w_0^{-1}uw_0h,s)=f_\tau^{V^{w_0},\psi_{w_0}}(h,s)$.
Thus, integrals \eqref{gl23} and \eqref{gl24} are equal to
\begin{equation}\label{gl25}
\int\limits_{Z({\bf A})T(F)U(B)({\bf A})\backslash G({\bf A})}\
\int\limits_{U(B)(F)\backslash U(B)({\bf A})}
\varphi_{\pi_1}^{U_1,\psi_{U_1}}(ug) R_\tau(g)dudg
\end{equation}
and
\begin{equation}\label{gl26}
\int\limits_{Z({\bf A})T(F)U(B)({\bf A})\backslash G({\bf A})}\
\int\limits_{U(B)(F)\backslash U(B)({\bf A})}
\varphi_{\pi_1}^{U_1,\psi_{U_1}}(ug)\varphi_{\pi_2}^{U_2,\psi_{U_2}}(ug)
R_\tau(g)dudg
\end{equation}
Here
$$R_\tau(g)=\int\limits_{V^{w_0}({\bf A})\backslash V({\bf
A})}f_\tau^{V^{w_0},\psi_{w_0}}(w_0vg,s)\psi_V(v)dv$$ and $T$ is
defined as all matrices of the form $T=\{\text{diag}(c,1)\ : c\in
F^*\}$. Consider first integral \eqref{gl25}. We apply lemma
\ref{fourgl1} part {\bf b)} to obtain that the integral
$$\int\limits_{U(B)(F)\backslash U(B)({\bf A})}
\varphi_{\pi_1}^{U_1,\psi_{U_1}}(ug)du$$ is a sum of terms which are
related to all unipotent orbits which are strictly greater than the
ones listed in that lemma part {\bf a)}, and to a certain constant
term. By cuspidality of $\pi_1$ we may ignore the summand with the
constant term. Also, the case we consider now corresponds to the
first column in Table 1. Thus ${\mathcal O}_{G_1}(\pi_1)$ consists
of the unipotent orbit specified in the first row of that Table.
From this, and from the computations done in the proof of lemma
\ref{fourgl1}, we obtain
$$\int\limits_{U(B)(F)\backslash U(B)({\bf A})}
\varphi_{\pi_1}^{U_1,\psi_{U_1}}(ug)du= \sum_{t\in
T(F)}L_{\pi_1}^{R_1}(tg)$$  Here $L_{\pi_1}$ is defined in the
beginning of section 2 and is given by
$$L_{\pi_i}(g_i)=\int\limits_{V_i(\pi_i)(F)\backslash V_i(\pi_i)({\bf
A})}\varphi_{\pi_i}(v_ig_i)\psi_{V_i(\pi_i)}(v_i)dv_i$$ where this
Fourier coefficient corresponds to the unipotent orbit as specified
in the first row in the second column of Table 1. Plugging the above
identity into integral \eqref{gl25}, collapsing summation with
integration, we deduce that in this case, integral \eqref{global1}
is a global unipotent integral.

Next consider integral \eqref{gl26}. Expand the function
$\varphi_{\pi_1}^{U_1,\psi_{U_1}}(ug)$ along the unipotent group
$U(B)$. By lemma \ref{fourgl1} part {\bf b)}, and by the cuspidality
of $\pi_1$, we may ignore the contribution from the constant term.
From part {\bf a)} of that lemma we obtain
\begin{equation}\label{expa1}
\varphi_{\pi_1}^{U_1,\psi_{U_1}}(ug)=\sum_{t\in T(F)}
\int\limits_{U(B)(F)\backslash U(B)({\bf A})}
\varphi_{\pi_1}^{U_1,\psi_{U_1}}(u_1tug)\psi_{U(B)}(u_1)du_1du
\end{equation}
Plug this into integral \eqref{gl26}. Use the fact that $\pi_2$ and
$f_\tau$ are left invariant by $T(F)$ to collapse summation and
integration. Thus we obtain
\begin{equation}\label{gl27}
\int\limits_{Z({\bf A})U(B)({\bf A})\backslash G({\bf A})}\
\int\limits_{U(B)(F)\backslash U(B)({\bf A})}
L^{R_1}_{\pi_1}(ug)\varphi_{\pi_2}^{U_2,\psi_{U_2}}(ug)
R_\tau(g)dudg
\end{equation}
From the above expansion we deduce that
$L^{R_1}_{\pi_1}(ug)=\psi_{U(B)}(u_1)L^{R_1}_{\pi_1}(g)$ Thus, using
lemma \ref{fourgl1} part {\bf a)},  integral \eqref{gl27} is equal
to
\begin{equation}\label{gl28}
\int\limits_{Z({\bf A})U(B)({\bf A})\backslash G({\bf A})}\
L^{R_1}_{\pi_1}(g)L^{R_2}_{\pi_2}(g) R_\tau(g)dudg
\end{equation}
where $L^{R_j}_{\pi_j}$ is defined as above and corresponds to the
unipotent orbits appearing in the first and second row of Table 2.
Thus, we deduce that also in this case integral \eqref{global1} is a
global unipotent integral.

Finally, we consider integral \eqref{gl21} when ${\mathcal
O}(E_\tau(\cdot,s))=((p+2)(p-2))$. Then the induction data is
${\mathcal O}(\tau_1)=(a(a-1))$ and ${\mathcal
O}(\tau_2)=((p-a+2)(p-a-1))$. Starting with integral \eqref{gl23},
we obtain
\begin{equation}\label{gl29}
\int\limits_{Z({\bf A})T(F)U(B)({\bf A})\backslash G({\bf A})}\
\int\limits_{U(B)(F)\backslash U(B)({\bf A})}
\varphi_{\pi_1}^{U_1,\psi_{U_1}}(ug) R_\tau(ug)dudg\notag
\end{equation}
Notice that in this case the function $R_\tau(g)$ is not left
invariant under $U(B)({\bf A})$. Now we consider the expansion given
by identity \eqref{expa1}. Using cuspidality of $\pi_1$ we may
ignore the contribution from the constant term, and collapsing
summation with integration, we obtain
\begin{equation}\label{gl210}
\int\limits_{Z({\bf A})U(B)({\bf A})\backslash G({\bf A})}\
\int\limits_{U(B)(F)\backslash U(B)({\bf A})} L^{R_1}_{\pi_1}(ug)
R_\tau(ug)dudg\notag
\end{equation}

As before, we have
$L^{R_1}_{\pi_1}(ug)=\psi_{U(B)}(u_1)L^{R_1}_{\pi_1}(g)$, and hence
we obtain
\begin{equation}\label{gl211}
\int\limits_{Z({\bf A})U(B)({\bf A})\backslash G({\bf A})}\
\int\limits_{U(B)(F)\backslash U(B)({\bf A})} L^{R_1}_{\pi_1}(g)
R_\tau^{U(B),\psi_{U(B)}}(g)dg\notag
\end{equation}
Using a variation of lemma \ref{fourgl1} part {\bf a)}, we obtain
that $R_\tau^{U(B),\psi_{U(B)}}(g)$ is a Fourier coefficient of the representation $\tau_1$ which
corresponds to the unipotent orbit $(a(a-1))$ of $GL_{2a-1}$, and a Fourier coefficient of $\tau_2$ which corresponds to
the unipotent orbit $((p-a+2)(p-a-1))$ of $GL_{2(p-a)+1}$. Hence, in
this case, integral \eqref{global1} is also a global unipotent
integral.

To summarize the above, we define the notion of an odd Eisenstein
series. We say that $E_\tau(\cdot,s)$, defined on $GL_{2p}$ is odd,
if, using induction by stages, there is a maximal parabolic subgroup
such that the value of $i$ as appear in lemma \ref{lem2} are odd, and
that $E_\tau(\cdot,s)$ is induced from that parabolic subgroup. Similarly, in
the other classical groups $H$, we define $E_\tau(\cdot,s)$ to be
odd if we can induce from a maximal parabolic subgroup whose Levi
part is $GL_\alpha\times L$ where $\alpha$ is odd. Here $L$ is a
classical group of the same type of $H$. We prove

\begin{theorem}\label{th1}
Assume that $G=GL_2$, and that all the groups $G_j$ are classical
groups. Then the global integral \eqref{global1} is a nonzero global
unipotent integral if and only if one of the representations
appearing in Tables 1 or 2, is an odd Eisenstein series.
\end{theorem}

\begin{proof}
The case when one of the Eisenstein series is odd was considered
above. Hence, we may assume that none of the Eisenstein series
appearing in integral \eqref{global1} is odd. As before we treat the
case where the Eisenstein series is defined on $H=GL_{2p}$. Assume
that $E_\tau(\cdot,s)$ is associated with the induced representation
$Ind_{P({\bf A})}^{H({\bf A})}\tau\delta_P^s$, where $P$ is the
maximal parabolic subgroup whose Levi part is $GL_{2r}\times
GL_{2(p-r)}$. Consider the Weyl element
$$w_0=\begin{pmatrix} &I_{2(p-r)}\\ I_{2r}&\end{pmatrix} $$
Unfolding the Eisenstein series, we consider the contribution to
integral \eqref{global1} from the double coset representative $w_0$.
A simple matrix conjugation implies that, as an inner integration,
we obtain an integral which involves the period integral
\begin{equation}\label{even1}
\int\limits_{Z({\bf A})G(F)\backslash G({\bf
A})}\varphi_{\pi_1}^{U_1,\psi_{U_1}}(g)\varphi_{\pi_2}^{U_2,\psi_{U_2}}(g)\ldots
\varphi_{\pi_{l-1}}^{U_{l-1},\psi_{U_{l-1}}}(g)\varphi_{\tau_1}^{U_r,\psi_{U_r}}(g)
\varphi_{\tau_2}^{U_{p-r},\psi_{U_{p-r}}}(g)dg
\end{equation}
Here, $U_r=U_{2r,2}$ the unipotent subgroup which was defined in
section 3, and $\psi_{U_r}$ is the character of this group as
defined in that section. Similarly for $U_{p-r}$. This process is
inductive. Namely, if any other representation is an Eisenstein
series, then by assumption it is not an odd Eisenstein series, and
we can further unfold it. Taking the right double coset
representative, similar to $w_0$, we obtain as inner integration
which is similar to the one given by integral \eqref{even1}.  From
this we conclude that eventually we will obtain a global integral
which involves a period integral of the type given by integral
\eqref{even1} where none of the representations is an Eisenstein series. This integral, if not zero, involves also integration
over a reductive group. Therefore, in this case, integral
\eqref{global1} is not a global {\sl unipotent} integral.

\end{proof}

To complete the classification for the group $G=GL_2$, we need to
consider integrals \eqref{gl21} and \eqref{gl22} where the
Eisenstein series $E_\tau(\cdot,s)$ is defined on the exceptional
group $GE_7({\bf A})$. We say that $E_\tau(\cdot,s)$ is an odd
Eisenstein series if the Levi part $M$ of the parabolic subgroup
from which we form the Eisenstein series, does not contain all three
roots $\alpha_2, \alpha_5$ and $\alpha_7$. In other words,
$E_\tau(\cdot,s)$ is odd if $M$ contains a subgroup of the type
$E_6$ or $A_6$ or $A_4\times A_2$. Here we label the roots of $GE_7$
as in \cite{G4}. In other words, $E_\tau(\cdot,s)$ is odd if $M$
does not contain the diagonal copy of $GL_2$ which stabilizes the
character $\psi_{U({\mathcal O})}$ as defined in section 3. Indeed,
from the definition of this character, it follows that this copy of
$GL_2$, contains the group $SL_2$ generated by
$x_{\alpha_2}(r)x_{\alpha_5}(-r)x_{\alpha_7}(r)$ and
$x_{-\alpha_2}(r)x_{-\alpha_5}(-r)x_{-\alpha_7}(r)$. We prove a
similar result to theorem \ref{th1}. We have
\begin{theorem}\label{th2}
Assume that $G=GL_2$, and that the Eisenstein series
$E_\tau(\cdot,s)$ is defined on the exceptional group $GE_7({\bf
A})$. Then the global integral \eqref{global1} is a nonzero global
unipotent integral if and only if the Eisenstein series appearing in
Tables 1 or 2, is an odd Eisenstein series.
\end{theorem}

\begin{proof}
The idea is the same as in the classical groups. First, if the
Eisenstein series is not odd, then a similar argument as in the
classical groups proves that the integral \eqref{global1} is not a
global unipotent integral. More precisely, assume that none of the
Eisenstein series appearing in integral \eqref{global1} is odd,
either on the classical groups or on $GE_7$. Then, it is not hard to
produce a Weyl element, which can be taken as a representative of
the double cosets $P\backslash H/V\cdot GL_2$, such that we obtain
an integral of the type of integral \eqref{even1}, as inner
integration. Thus we conclude that integral \eqref{global1} is not a
nonzero global unipotent integral.

Next we consider the case where the Eisenstein series
$E_\tau(\cdot,s)$ is an odd Eisenstein series defined on $GE_7({\bf
A})$. Thus there are three cases to consider. In of each of them, we
first write a certain Weyl element $w_0$, which will be the only
double coset representative of $P\backslash H/V\cdot GL_2$ which
will contribute a nonzero term in the unfolding process. For that
element we also write down the group $w_0^{-1} (V\cdot GL_2) w_0\cap
M$. To obtain $w_0$, we first write down $w_1$ which is the shortest
Weyl element in $M\backslash GE_7$. Then we consider the Weyl
element $w_1w[257]$ where $w[257]$ is the unique reflection in the
group $GL_2$ as embedded above in $GE_7$. Let $w_0$ denote the
shortest Weyl element which is in the same coset $M\backslash GE_7$
as $w_1w[257]$. Thus, ignoring the contributions to integral
\eqref{global1} from the other terms, we then consider the integral
\begin{equation}\label{e71}
\int\limits_{Z({\bf A})B(F)\backslash G({\bf
A})}\varphi_{\pi_1}^{U_1,\psi_{U_1}}(g)\varphi_{\pi_2}^{U_2,\psi_{U_2}}(g)
\int\limits_{V^{w_0}({\bf A})\backslash V({\bf
A})}f_\tau^{V^{w_0},\psi_{w_0}}(w_0vg,s)\psi_V(v)dvdg
\end{equation}
Here $B$ is the Borel subgroup of $G=GL_2$ which consists of all
upper unipotent matrices, and
$$f_\tau^{V^{w_0},\psi_{w_0}}(h,s)=
\int\limits_{V^{w_0}(F)\backslash V^{w_0}({\bf
A})}f_\tau(v_0h,s)\psi_{w_0}(v_0)dv_0$$ 
Integral \eqref{e71} corresponds to the case of integral \eqref{gl22}. However, if we assume that $\varphi_{\pi_2}^{U_2,\psi_{U_2}}(g)$ is one for all $g$, then it covers also the case of integral \eqref{gl21}. Notice that this is exactly as in integral
\eqref{gl24} where we considered the classical groups. To complete
the study of these cases, we need to determine the groups
$V^{w_0}\cap M$ and the sets ${\mathcal O}_M(\tau)$. Then, using
lemma \ref{fourgl1}, we argue in a similar way as in the the case of
the classical groups. To determine the unipotent orbit ${\mathcal
O}_M(\tau)$ we use the dimension identity $\text{dim}\ E_\tau=
\text{dim}\ \tau+\text{dim}\ U(P)$ established in lemma \ref{lemex}.
Here $U(P)$ is the unipotent radical of $P$.  It follows from Tables
1 and 2 that ${\mathcal O}_{GE_7}(E_\tau)$ is $E_7(a_2)$ or
$E_7(a_1)$. Hence, $\text{dim}\ E_\tau= 61,62$. From this it is easy
to determine $\text{dim}\ \tau$, and hence to determine ${\mathcal
O}_M(\tau)$. There are three  types of odd Eisenstein series, and we
consider each one of them.

{\bf 1)}\ Suppose that $M$ is of type $A_6$. Then $\text{dim}\
U(P)=42$, and hence $\text{dim}\ \tau=19$ if ${\mathcal
O}_{GE_7}(E_\tau)=E_7(a_2)$ and $\text{dim}\ \tau=20$ if ${\mathcal
O}_{GE_7}(E_\tau)=E_7(a_1)$. Hence ${\mathcal O}_M(\tau)=(52)$ in
the first case and ${\mathcal O}_M(\tau)=(61)$ in the second case.
As for the group $V^{w_0}\cap M$ in this case, it is defined as
follows. Let $U$ denote the maximal unipotent subgroup of $GL_7$.
Then $V^{w_0}\cap M=\{ u\in U\ : u_{1,2}=u_{3,4}=0\}$. The character
$\psi_{w_0}$ is defined as
$$\psi_{w_0}(v)=\psi(v_{1,3}+v_{2,4}+v_{4,5}+v_{5,6}+v_{6,7})$$ The
corresponding Fourier coefficient is associated with the unipotent
orbit $(52)$ of $GL_7$.

{\bf 2)}\ Suppose that $M$ is of type $E_6$. In this case we have
$\text{dim}\ \tau=34,35$, and hence ${\mathcal
O}_M(\tau)=D_5,E_6(a_1)$. The group $V^{w_0}\cap M$ in this case is
defined as follows. Let $Q$ denote the parabolic subgroup of $GE_7$
whose Levi part is $T(GE_7)\cdot(SL_2\times SL_2)$ where the group
$SL_2\times SL_2$ is generated by $x_{\pm\alpha_2}$ and
$x_{\pm\alpha_3}$. Also, $T(GE_7)$ is the maximal torus of $GE_7$.
Let $U(Q)$ denote the unipotent radical of $Q$. Then $V^{w_0}\cap
M=U(Q)$. The character $\psi_{w_0}$ in the case is defined as
follows. For $u\in U(Q)$, write
$$u=x_{\alpha_1}(r_1)x_{\alpha_3+\alpha_4}(r_2)x_{\alpha_2+\alpha_4}(r_3)
x_{\alpha_5}(r_4)x_{\alpha_6}(r_5)u'$$ Here $u'$ is an element in
$U(Q)$ which is a product of one dimensional unipotent subgroups of
$U(Q)$ corresponding to positive roots of $E_7$ and does not include
any one of the above five roots. Then, we define
$\psi_{w_0}(u)=\psi(r_1+r_2+\cdots +r_5)$. It is not hard to check
that the corresponding Fourier coefficient
$f_\tau^{V^{w_0},\psi_{w_0}}$ is associated to the unipotent orbit
$D_5$ of $E_7$.

{\bf 3)}\ Suppose that $M$ is of type $A_4\times A_2$. In this case
$\text{dim}\ \tau=11,12$. The unipotent group $V^{w_0}\cap M$,
viewed as a subgroup of $GL_5\times GL_3$ is defined as follows. Let
$U_5$ denote the standard maximal unipotent subgroup of $GL_5$, and
similarly define $U_3$. Define $V_5=\{ v\in U_5\ : u_{1,2}=0\}$ and
$V_3=\{ v\in U_3\ : u_{1,2}=0\}$. Then we have $V^{w_0}\cap
M=V_5\times V_3$. The character $\psi_{w_0}$, is then a product of
$\psi_{w_0,5}$ and $\psi_{w_0,3}$ defined on the groups $V_5$ and
$V_3$. Here, $\psi_{w_0,5}(v)=\psi(v_{2,3}+v_{3,4}+v_{4,5})$ and
$\psi_{w_0,3}(v)=\psi(v_{2,3})$. Thus, on $GL_5$ this Fourier
coefficient corresponds to the unipotent orbit $(41)$, and on $GL_3$
it corresponds to the orbit $(21)$. From this we can determine the
sets ${\mathcal O}(\tau_i)$. If $\text{dim}\ \tau=11$, then the only
option is ${\mathcal O}(\tau_1)=(41)$ and ${\mathcal
O}(\tau_2)=(21)$. There are other cases with $\text{dim}\ \tau=11$
which we ignore since at least one of the sets ${\mathcal
O}(\tau_i)$ does not support $f_\tau^{V^{w_0},\psi_{w_0}}$. When
$\text{dim}\ \tau=12$ there are two options. The first option is
${\mathcal O}(\tau_1)=(5)$ and ${\mathcal O}(\tau_2)=(21)$, and the
second option is ${\mathcal O}(\tau_1)=(41)$ and ${\mathcal
O}(\tau_2)=(3)$.

Next we proceed as in the case of the classical groups. Assume first
that ${\mathcal O}_{GE_7}(E_\tau)=E_7(a_2)$. Then we obtain that
$f_\tau^{V^{w_0},\psi_{w_0}}(uh,s)=f_\tau^{V^{w_0},\psi_{w_0}}(h,s)$
for all $u\in U(B)({\bf A})$, and now we proceed exactly as in
integrals \eqref{gl25} and \eqref{gl26}. In the second case, when
${\mathcal O}_{GE_7}(E_\tau)=E_7(a_1)$ we proceed exactly as with
the case of  ${\mathcal O}(E_\tau(\cdot,s))=((p+2)(p-2))$ in the
classical groups. See right after integral \eqref{gl28}.

Finally, we need to analyze the contribution to the unfolding
process from other double cosets representatives of $P\backslash
H/V\cdot GL_2$. We need to show that all of them contribute zero to
the global integral. The process of doing it is similar to the one
carried out in details for $H=GL_{2p}$. We omit the details of this
computation.

\end{proof}


\section{\bf Proof of Lemma \ref{lem1}}

In this section we prove lemma \ref{lem1}. We will consider the case
of $D_5$ in details. The case of $D_5(a_1)$ is similar. Let $\pi$
denote an irreducible cuspidal representation of $GE_6({\bf A})$. We
assume that ${\mathcal O}(\pi)=D_5$ and derive a contradiction. We
describe the Fourier coefficient associated with this unipotent
orbit. Let $P=MU$ denote the parabolic subgroup of $GE_6$ whose Levi
part is $M=T\cdot (SL_2\times SL_2)$. Here $T$ is the maximal torus
of $GE_6$ and the two copies of $SL_2$ are generated by $x_{\pm
001000};\ x_{\pm 000010}$. Consider the group $U/[U,U]$. As coset
representatives we may choose the one parameter subgroups $x_\alpha$
where $\alpha$ is one of the following nine roots
$$(100000);\ (101000);\ (000001);\ (000011);\ (000100);\ (001100);\
(000110);\ (001110);\ (010000)$$ The group $M$ acts on these
representatives as follows. On the first two it acts, up to a power
of the determinant, as the standard representation of $GL_2$ which
contains the $SL_2$ generated by $x_{\pm 001000}$. On the next two
representatives its acts similarly, but this time the $GL_2$
contains the group generated by $x_{\pm 000010}$. On the next four
$M$ acts as the tensor product of $GL_2\times GL_2$, and on the last
representatives, it acts as a one dimensional representation. From
this we can define the corresponding Fourier coefficient. Given
$u\in U$ write
$$u=x_{100000}(r_1)x_{001100}(r_2)x_{000110}(r_3)x_{000011}(r_4)x_{010000}(r_5)u'$$
Here $u'\in U$ is any product of one parameter subgroups associated
with positive roots of $E_6$ which do not include the above five
roots. Denote $\psi_U(u)=\psi(r_1+r_2+\cdots + r_5)$ and define the
Fourier coefficient associated with the unipotent orbit $D_5$ by
\begin{equation}\label{fourier1}
\int\limits_{U(F)\backslash U({\bf A})}\varphi_\pi(u)\psi_U(u)du
\end{equation}
The assumption that ${\mathcal O}(\pi)=D_5$ asserts that this
Fourier coefficient is not zero for some choice of data, but any
Fourier coefficient of $\pi$ associated with the unipotent orbits
$E_6$ or $E_6(a_1)$ is zero for all choice of data.

For $1\le i\le 6$, let $w_i$ denote the simple reflection associated
with the root $\alpha_i$. Let $w_0=w_6w_5w_4w_3w_2w_4w_5w_1w_3$. We
have
$$w_0 \alpha_1=\alpha_2;\ w_0 (001100)=\alpha_4;\ w_0 (000110)=\alpha_1;\
w_0 (000011)=\alpha_5;\ w_0 \alpha_2=\alpha_3$$ Conjugating by
$w_0$, the above Fourier coefficient is equal to
$$\int\limits_{V^-(F)\backslash V^-({\bf A})}
\int\limits_{V^+(F)\backslash V^+({\bf A})}
\int\limits_{U(D_5)(F)\backslash U(D_5)({\bf
A})}\varphi_\pi(uv^+v^-w_0)\psi_{U(D_5)}(u)dudv^+dv^-$$ where the
notations are as follows. First, the group $U(D_5)$ is the maximal
unipotent subgroup of type $D_5$ generated by the simple roots
$\alpha_i$ for $1\le i\le 5$. The character $\psi_{U(D_5)}$ is the
Whittaker coefficient defined on $U(D_5)$. The group $V^+$ consists
of all unipotent elements $x_\alpha$ where $\alpha$ is one of the
roots
$$ (111211);\ (011221);\ (112211);\ (111221);\ (112221);\ (112321);\
(122321)$$ Similarly, the group $V^-$ is defined by all
$x_{-\alpha}$ where $\alpha$ is one of the roots
$$(101111);\ (011111);\ (001111);\ (010111);\ (000111);\ (000011);\
(000001)$$ Thus, by definition, the above integral is not zero for
some choice of data.

We expand the above integral along the one parameter subgroup
$x_{111111}(r)$. Thus, the above integral is equal to
$$\int\sum_{\xi\in F}\int\limits_{F\backslash {\bf A}}\int
\varphi_\pi(ux_{111111}(r)v^+v^-w_0)\psi_{U(D_5)}(u)\psi(\xi
r)dudrdv^+dv^-$$ Conjugate from left to right by the element
$x_{-(101111)}(-\xi)$. Changing variables, first in $U(D_5)$ and
then in $V^+$ we obtain that the integral
$$\int\limits_{V_1^-(F)\backslash V_1^-({\bf A})}
\int\limits_{V_1^+(F)\backslash V_1^+({\bf A})}
\int\limits_{U(D_5)(F)\backslash U(D_5)({\bf
A})}\varphi_\pi(uv^+v^-)\psi_{U(D_5)}(u)dudv^+dv^-$$ is not zero for
some choice of data. Here $V_1^+$ consists of all elements
$x_{\alpha}$ in $V^+$ including  $x_{111111}$. The group $V_1^-$
consists of all $x_\alpha$ in $V^-$ without the root $-(101111)$.
Thus $\text{dim}\ V_1^+=\text{dim}\ V^+ +1$ and $\text{dim}\
V_1^-=\text{dim}\ V^- -1$.

Proceed with this expansion four more times. First expand along
$x_{011211}$ and use the element $x_{-(011111)}$. Then expand along
$x_{101111}$ and use $x_{-(001111)}$, then expand along $x_{011111}$
and use $x_{-(010111)}$, and finally expand along $x_{001111}$ and
use $x_{-(000111)}$. We deduce that the integral
$$\int\limits_{V_2^-(F)\backslash V_2^-({\bf A})}
\int\limits_{V_2^+(F)\backslash V_2^+({\bf A})}
\int\limits_{U(D_5)(F)\backslash U(D_5)({\bf
A})}\varphi_\pi(uv^+v^-)\psi_{U(D_5)}(u)dudv^+dv^-$$ is not zero for
some choice of data. To describe the notations in the above
integral,  let $R$ denote the unipotent radical of the maximal
parabolic subgroup of $E_6$ whose Levi part contains $Spin_{10}$
which contains the group $U(D_5)$. Thus $R$ is the abelian group
generated by all $x_\alpha$ such that
$\alpha=\sum_{i=1}^5n_i\alpha_i +\alpha_6$. Then $V_2^+$ consists of
all $x_\alpha\in R$  not including the roots $(000001);\ (000011);\
(000111);\ (010111)$. Thus $\text{dim}\ V_2^+=\text{dim}\ R-4=12$.
The group $V_2^-$ consists of all $x_{-\alpha}$ such that $\alpha$
is one of the two roots $(000011);\ (000001)$.

Next we expand the above integral along the unipotent subgroup
$x_{010111}(r)$. Consider first the contribution from the non
trivial character. We claim that it contributes zero to the
expansion. Indeed, in this case after a conjugation by the Weyl
element $w_5w_6$ it is not hard to check that we obtain the Fourier
coefficient of $\pi$ which is associated with the unipotent orbit
$E_6(a_1)$. By the assumption that ${\mathcal O}(\pi)=D_5$, we
deduce that this Fourier coefficient is zero. Thus, we are left with
the contribution from the trivial character, and we obtain that the
integral
$$\int\limits_{V_2^-(F)\backslash V_2^-({\bf A})}
\int\limits_{V_3^+(F)\backslash V_3^+({\bf A})}
\int\limits_{U(D_5)(F)\backslash U(D_5)({\bf
A})}\varphi_\pi(uv^+v^-)\psi_{U(D_5)}(u)dudv^+dv^-$$ is not zero for
some choice of data. Here $V_3^+$ is the group generated by $V_2^+$
and $x_{010111}$. Now we expand along $x_{000111}$ and as before we
use the element $x_{-(000011)}$ and then expand along $x_{000011}$
and use $x_{-(000001)}$. Thus we obtain that the integral
$$\int\limits_{V_4^+(F)\backslash V_4^+({\bf A})}
\int\limits_{U(D_5)(F)\backslash U(D_5)({\bf
A})}\varphi_\pi(uv^+)\psi_{U(D_5)}(u)dudv^+$$ is not zero for some
choice of data. Here $V_4^+$ is the group generated by $V_3^+$ and
$x_{000111}$ and $x_{000011}$. Finally, we expand along
$x_{000001}$. The  contribution from the nontrivial orbit gives
zero. Indeed, in this case we obtain the Fourier coefficient of
$\pi$ associated with the unipotent orbit $E_6$. As argued above, it
is zero. The contribution from the trivial orbit  also contributes
zero. Indeed, in this case we obtain as an inner integration, the
constant term along the unipotent radical $R$. By cuspidality of
$\pi$ it is zero. Thus the above integral is zero and we derived a
contradiction.

The case when ${\mathcal O}(\pi)=D_5(a_1)$ is similar. We give some
details. Let $U'$ denote the unipotent radical of the parabolic
subgroup of $E_6$ whose Levi part is $T\cdot SL_2$ where the $SL_2$
is generated by $x_{\pm \alpha_4}$. Thus $\text{dim}\ U'=35$. Let
$U$ be the subgroup of $U'$ where we omit the 3 unipotent elements
$x_{001100};\ x_{000010};\ x_{000110}$. Thus $\text{dim}\ U=32$.
Given $u\in U$ write
$$u=x_{010000}(r_1)x_{101100}(r_2)x_{000011}(r_3)x_{000111}(r_4)x_{001110}(r_5)u'$$
where $u'\in U$ is an element generated by all $x_\alpha$ such that
$\alpha$ is not one of the above five roots. Define
$\psi_U(u)=\psi(r_1+r_2+\cdots + r_5)$. Then we can form the
corresponding Fourier coefficient given by the integral
\eqref{fourier1}. Let $w_0=w_6w_5w_4w_3w_2w_4w_5w_1$. We have
$w_0(010000)=\alpha_3;\ w_0(010100)=(001100);\
w_0(101100)=(010100);\ w_0(000011)=\alpha_5;\ w_0(001110)=\alpha_1$.
The next step is to expand the integral, and use the fact ${\mathcal
O}(\pi)=D_5(a_1)$. Eventually, we obtain as inner integration, a
constant term along a certain unipotent radical, which is zero by
cuspidality. We omit the details.

\end{document}